\documentclass[11pt, reqno]{amsart}
\usepackage{amscd}        

\usepackage[english]{babel}
\usepackage[utf8]{inputenc}
\usepackage{amsmath, amssymb}
\usepackage{enumitem}
\usepackage{enumitem}
\usepackage{booktabs}
\usepackage{comment}

\usepackage{booktabs}
\usepackage[svgnames, dvipsnames, table]{xcolor}
\definecolor{cornflowerblue}{HTML}{9ACEEB}
\definecolor{cornflowerblue-complement}{HTML}{2A96D2}
\usepackage[colorlinks = true, allcolors = cornflowerblue-complement, pagebackref=true, colorlinks, pdfencoding=auto]{hyperref}

\usepackage{tabu}

\usepackage{siunitx} 

\usepackage{orcidlink}

\usepackage{pdflscape}

\usepackage{float}

\usepackage{nomencl}

\usepackage{algorithm}
\usepackage{algpseudocode}

\usepackage{multirow}

\usepackage{tikz-cd}
\usepackage{graphicx}
\usepackage{rotating}
\usepackage{diagbox}
\usepackage{amssymb}
\usepackage{epstopdf}
\usepackage{tikz}
\usepackage{pgfplots}
\pgfplotsset{compat=1.18} 
\definecolor{UCRceleste}{RGB}{0,192,243}
\definecolor{Crimson}{RGB}{220, 20, 60}
\definecolor{Turquoise}{HTML}{04d4f0}
\definecolor{GunmetalGray}{HTML}{1ca6a0}
\definecolor{HotPink}{HTML}{f652a0}
\definecolor{BlueGrotto}{HTML}{059dc0}
\definecolor{GaloisBlue}{HTML}{6495ED}
\definecolor{GaloisRed}{HTML}{DE3163}
\definecolor{mintgreen}{RGB}{152,255,152}
\definecolor{pinksalmon}{RGB}{255,102,102}
\definecolor{hueso}{RGB}{245,245,220}
\definecolor{marfil}{RGB}{255,253,208}
\definecolor{amarillo}{RGB}{255,255,0}
\usetikzlibrary{decorations.markings,arrows}
\usetikzlibrary{decorations.pathreplacing}
\usepackage[inner=1.0in,outer=1.0in,bottom=1.0in, top=1.0in]{geometry}

\numberwithin{equation}{section}

\newtheorem{theorem}{Theorem}[section]
\newtheorem{lemma}[theorem]{Lemma}
\newtheorem{proposition}[theorem]{Proposition}
\newtheorem{corollary}[theorem]{Corollary}

\makeatletter
\def\moverlay{\mathpalette\mov@rlay}
\def\mov@rlay#1#2{\leavevmode\vtop{%
   \baselineskip\z@skip \lineskiplimit-\maxdimen
   \ialign{\hfil$\m@th#1##$\hfil\cr#2\crcr}}}
\newcommand{\charfusion}[3][\mathord]{
    #1{\ifx#1\mathop\vphantom{#2}\fi
        \mathpalette\mov@rlay{#2\cr#3}
      }
    \ifx#1\mathop\expandafter\displaylimits\fi}
\makeatother

\newcommand{\hooklongrightarrow}{\lhook\joinrel\longrightarrow}

\newcommand{\suchthat}{\;\ifnum\currentgrouptype=16 \middle\fi|\;}

\newcommand{\Z}{\mathbb{Z}}
\newcommand{\C}{\mathbb{C}}
\newcommand{\Q}{\mathbb{Q}}
\newcommand{\R}{\mathbb{R}}
\newcommand{\Gal}[1]{\operatorname{Gal}#1}
\newcommand{\Cl}{\operatorname{Cl}}
\newcommand{\op}[1]{\operatorname{#1}}

\newcommand{\ord}[1]{\operatorname{ord}#1}

\newcommand{\Nrm}[2]{\mathbf{N}_{#1}(#2)}

\theoremstyle{definition}
\newtheorem{remark}[theorem]{Remark}
\newtheorem{definition}[theorem]{Definition}
\newtheorem{example}[theorem]{Example}

\usepackage{listings}
\definecolor{sagebrown}{RGB}{176, 92, 10}
\definecolor{sageblue}{RGB}{44, 45, 254}
\definecolor{sagepurple}{RGB}{151, 57, 164}
\definecolor{sagegreen}{RGB}{18, 103, 68}
\definecolor{sagered}{RGB}{170, 16, 15}

\lstdefinestyle{SageMath-style}{
    backgroundcolor=\color{white},   
    commentstyle=\color{sagebrown},
    keywordstyle=\color{sagepurple},
    keywordstyle = [2]{\color{sageblue}},
    keywordstyle = [3]{\color{yellow}},
    numberstyle=\tiny\color{sagegreen},
    stringstyle=\color{sagered},
    basicstyle=\ttfamily\footnotesize,
    breakatwhitespace=false,         
    breaklines=true,                 
    captionpos=b,                    
    keepspaces=true,                 
    numbers=left,                    
    numbersep=5pt,                  
    showspaces=false,                
    showstringspaces=false,
    showtabs=false,                  
    tabsize=2
}

\begin{document}

\title{The distribution of $k$-free ideals in ray class groups}

\author[]{Adrian Barquero-Sanchez\orcidlink{0000-0001-7847-2938}, Jack Heimrath\orcidlink{0000-0002-9334-4667}, Bernd Sing\orcidlink{0000-0001-7700-9976}, Nicolás Sirolli\orcidlink{0000-0002-0603-4784}, Caylee Spivey\orcidlink{0009-0000-3178-9188} and Michael Wijaya}

\address{Escuela de Matem\'atica, Universidad de Costa Rica, San Jos\'e 11501, Costa Rica}
\address{Centro de Investigación en Matemática Pura y Aplicada, Universidad de Costa Rica, San Jos\'e 2060, Costa Rica}
\email{adrian.barquero\_s@ucr.ac.cr}

\address{Department of Mathematics and Statistics, Queen’s University, Kingston, Ontario, Canada}
\email{23DK4@queensu.ca}

\address{Department of Mathematics, University of the West Indies, Cave Hill, P.O. Box 64, Bridgetown, BB11000, Barbados}
\email{bernd.sing@uwi.edu}

\address{Departamento de Matemática, FCEyN, UBA and IMAS, CONICET -- Pabellón I, Ciudad Universitaria, Ciudad Autónoma de Buenos Aires (1428), Argentina}
\email{nsirolli@dm.uba.ar}

\address{Department of Mathematics, University of Connecticut, Storrs, CT 06268, USA}
\email{caylee.spivey@uconn.edu}

\address{Department of Mathematics, Trinity School, New York, NY 10024, USA}
\email{michael.wijaya@trinityschoolnyc.org}


\begin{abstract}
In this paper, we extend the classical problem of studying the distribution of $k$-free integers in arithmetic progressions to the setting of arbitrary number fields. Using the language of ray class groups, we establish asymptotic formulas, together with error terms, for the number of relatively $k$-free $r$-tuples of integral ideals of bounded norm lying in prescribed ray classes. In particular, specializing to $r=1$ yields an asymptotic formula for the number of $k$-free ideals in a fixed ray class and shows that such ideals are equidistributed among ray classes.

Our error terms are formulated in terms of an admissible exponent for the corresponding ideal-counting problem, allowing sharper ideal-counting estimates to be incorporated directly. In particular, the stronger estimates available for ordinary and narrow ideal class groups yield improved error terms. Our results recover the classical formulas of Gegenbauer and Cohen--Robinson over $\mathbb{Q}$ and extend Sittinger's theorem on relatively $k$-free tuples of ideals to ray classes. We also present explicit computational examples illustrating the asymptotic formulas and the equidistribution of $k$-free ideals among ray classes.
\end{abstract}

\maketitle

\section{Introduction and statement of results}

\subsection{History: The distribution of \texorpdfstring{$k$}{k}-free integers}
An integer $n$ is said to be $k$-free if it is not divisible by the $k$-th power of any prime number. A classical problem in number theory concerns the study of the distribution of $k$-free integers under various conditions. The first known occurrence of this topic in general appears in a result of Gegenbauer \cite[pp.~46--47]{Gege85}, who proved in 1885 that, for $k \ge 2$,
\begin{align}\label{eqn:gegenbauer-formula}
\#\{ n \in \Z_{\ge 1} \suchthat \text{$n \le X$ and $n$ is $k$-free} \}
= \frac{1}{\zeta(k)} \cdot X + O(X^{1/k}),
\end{align}
where $\zeta(s)$ denotes the Riemann zeta function. In fact, Gegenbauer mentions that the special case $k = 2$ had already been derived by Dirichlet, Bugajef, and Cesàro \cite[p. 47]{Gege85}.  

Just as Dirichlet studied the distribution of primes in arithmetic progressions, another natural question that has been investigated is how $k$-free numbers are distributed among arithmetic progressions. The case of square-free numbers ($k = 2$) has been the most extensively analyzed. Early contributions in this direction were due to Prachar \cite{Pra58} and Erd\H{o}s \cite{Erd60}, who studied the problem of determining the least square-free number in a given arithmetic progression $a \pmod{q}$. 

A systematic investigation of the \emph{distribution} of square-free integers in arithmetic progressions began with the work of Landau \cite[pp. 633--636]{landau1909handbuch}, who obtained the asymptotic formula
\[
\#\{ n \in \Z_{\ge 1} \suchthat \text{$n \le X, \, n \equiv a \bmod q$ and $n$ is square-free} \}
= \frac{6}{q\pi^2}
\prod_{p\mid q} \left(1 - \tfrac{1}{p^2} \right)^{-1} \cdot X
+ O(X^{1/2})
\]
for $\gcd(a, q) = 1$. This formula shows, in particular, that square-free integers are equidistributed among arithmetic progressions.

Orr \cite{Orr71} obtained remainder estimates and mean-square estimates for the error terms in counting square-free integers in arithmetic progressions, averaged over residue classes and moduli $q$, in analogy with results of Barban, Bombieri, Davenport, Halberstam, and Gallagher for primes in arithmetic progressions, while Hooley \cite{Hoo75} sharpened Prachar’s error term for square-free integers in arithmetic progressions, obtaining uniform asymptotic estimates in a wider range of moduli under suitable hypotheses on $q$.  These works mark the classical foundation of the distribution theory of square-free numbers in arithmetic progressions.

The extension to general $k$-free integers ($k \ge 2$) follows a similar structure. Cohen and Robinson \cite{cohen1963distribution} showed that the same asymptotic form (note that $1/\zeta(2) = 6/\pi^2$)
\begin{equation}
\label{eqn:cohen}
\#\{ n \in \Z_{\ge 1} \suchthat \text{$n \le X, \, n \equiv a \bmod q$ and $n$ is $k$-free} \}
= \frac{1}{q\,\zeta(k)}
  \prod_{p\mid q}\left(1 - \frac{1}{p^k}\right)^{-1}
  \cdot X
  + O(X^{1/k})
\end{equation}
holds in this more general setting.
Subsequent works, such as those of Baker and Powell \cite{BP10}, Liu \cite{Liu14,Liu16}, and Mossinghoff, Oliveira e Silva, and Trudgian \cite{MOT21}, have refined both the theoretical framework and computational aspects of the distribution of $k$-free numbers.  

More recently, Nunes \cite{Nun15,Nun22}, Mangerel \cite{Man21}, and Gorodetsky, Matomäki, Radziwiłł, and Rodgers \cite{GMRR21} have investigated finer aspects of the variance and moments of the distribution of square-free and $k$-free integers in short intervals and arithmetic progressions. Their work goes beyond classical asymptotic formulas by studying variances, moments, and other statistical features of the error terms. In this way, the study of $k$-free integers in arithmetic progressions has become closely connected with contemporary developments in the analytic and probabilistic theory of multiplicative functions; see also \cite{montgomery-vaughan2007}.

\medskip

A related but distinct direction concerns simultaneous coprimality conditions for tuples. Mertens considered the probability that two positive integers are relatively prime. Lehmer \cite{lehmer1900asymptotic} extended this problem to $r$-tuples, showing that the probability that $r$ positive integers have greatest common divisor equal to one is $1/\zeta(r)$. More generally, an $r$-tuple $(n_1,\dots,n_r)$ of positive integers is said to be \emph{relatively $k$-free} if there is no prime number $p$ such that $p^k$ divides each of $n_1,\dots,n_r$. Benkoski \cite{benkoski1976probability} extended Lehmer's result by proving that, if $r \neq 1$ and $rk>2$, then
\begin{align}\label{eqn:benkoski-formula}
\#\left\{
(n_i)\in \Z_{\geq 1}^r
\suchthat
n_i \leq X \text{ for all } i
\text{ and } (n_i) \text{ is relatively $k$-free}
\right\} = 
\frac{1}{\zeta(kr)}X^r + O(X^{r-1}).
\end{align}
Nymann \cite{nymann1992distribution} studied the equidistribution of relatively $k$-free tuples among residue classes. 

The number-field analogues of these results, due to Sittinger \cite{Sit10} and to Sittinger and DeMoss \cite{Sit20}, extend Benkoski's asymptotic formula to tuples of integral ideals. They are especially important for the present paper and served as a key source of motivation for our work. In both papers, however, the tuples are counted among all integral ideals of bounded norm, with no condition imposed on their ideal classes. The present paper continues this number-field perspective, but focuses on the distribution of $k$-free integral ideals in ray classes, thereby extending the classical study of $k$-free integers in arithmetic progressions to the setting of arbitrary number fields, and establishing the equidistribution of such ideals among the ray classes.

\subsection{Statement of results}

Let $K$ be a number field. We say that a nonzero integral ideal
$\mathfrak{a}\trianglelefteq\mathcal{O}_K$ is \emph{$k$-free} if
there is no nonzero prime ideal $\mathfrak{p}$ of $\mathcal{O}_K$
such that $\mathfrak{p}^k$ divides $\mathfrak{a}$.

In this paper, we study the distribution of $k$-free integral ideals
of $K$ in ray classes. Ray class groups provide the natural
number-field analogue of arithmetic progressions. This analogy is
made explicit in the paragraph preceding
Corollary~\ref{cor:gegenbauer-cohen-robinson} and in
Example~\ref{ex:ray-class-groups-over-Q}. For a modulus
$\mathfrak{m} = \mathfrak{m}_0 \mathfrak{m}_{\infty}$, we denote the ray class group modulo $\mathfrak{m}$
by $\mathrm{Cl}_K^{\mathfrak{m}}$. We refer the reader to
Section~\ref{section:ray-class-groups} for the relevant definitions
and basic properties of ray class groups and related objects.

For a ray class $\mathfrak{A}\in\mathrm{Cl}_K^{\mathfrak{m}}$ and
$X>0$, we write $\mathcal{I}_K^{\mathfrak{m}}$ for the set of all
integral ideals of $K$ that are coprime to the integral ideal
$\mathfrak{m}_0$, and define
\begin{equation}\label{eqn:ideals-in-ray-class-bounded-norm}
\mathcal{I}_K^{\mathfrak{m}}(X;\mathfrak{A})
:=
\left\{
\mathfrak{a}\in\mathcal{I}_K^{\mathfrak{m}}
\suchthat
\Nrm{K/\Q}{\mathfrak{a}}\leq X
\text{ and }
[\mathfrak{a}]=\mathfrak{A}
\right\}.
\end{equation}
We denote the corresponding ideal-counting function by
\begin{equation}\label{eqn:ray-class-ideal-counting-function}
\mathcal{H}_K^{\mathfrak{m}}(X;\mathfrak{A})
:= \#\mathcal{I}_K^{\mathfrak{m}}(X;\mathfrak{A}).
\end{equation}

In order to state our main theorem in a form that makes its dependence
on ideal-counting estimates explicit, we isolate the analytic input
that enters our argument. Let $0<\delta\leq1$. We say that $\delta$ is
an \emph{admissible ideal-counting exponent} for the pair
$(K,\mathfrak{m})$ if, for every ray class
$\mathfrak{A}\in\Cl_K^{\mathfrak{m}}$,
\begin{equation}\label{eqn:admissible-ideal-counting-introduction}
\mathcal{H}_K^{\mathfrak{m}}(X;\mathfrak{A})
= \rho_K^{\mathfrak{m}}X + O\left(X^{1-\delta}\right)
\end{equation}
as $X\to\infty$. Since the ray class group is finite, the implied
constant may be chosen uniformly over
$\mathfrak{A}\in\Cl_K^{\mathfrak{m}}$.

The classical estimates recalled in Section~\ref{section:ideal-counting-functions}
show that $\delta=1/d$ is admissible for every modulus of a number field of
degree $d$. A stronger exponent is available for the ordinary and narrow class
groups. More precisely, let
\begin{align}\label{eqn:narrow-class-group-modulus}
\mathfrak{m}_{\infty}^{\mathrm{full}} := \prod_{\substack{v\mid\infty\\ v\text{ real}}} v
\end{align}
denote the modulus whose finite part is trivial and whose infinite part
contains every real place of $K$ with exponent $1$. When $d>1$, if
$\mathfrak{m}=(1)$ or
$\mathfrak{m}=\mathfrak{m}_{\infty}^{\mathrm{full}}$, then
$\delta=2/(d+1)$ is admissible. For $K=\Q$, the general choice
$\delta=1/d=1$ gives an $O(1)$ ideal-counting error. We formulate our main
result for an arbitrary admissible exponent $\delta$. In this way, any sharper
ideal-counting estimate available in a particular setting can be inserted
directly into the theorem.

We can now state our main result, which counts relatively $k$-free
tuples of ideals lying in prescribed ray classes.

\begin{theorem}\label{thm:main-theorem}
Let $K$ be a number field of degree $[K:\Q]=d$, let $\mathfrak{m}$
be a modulus for $K$, and let $\delta$ be an admissible ideal-counting
exponent for $(K,\mathfrak{m})$. Let $k,r\in\Z_{\geq1}$ satisfy $kr>1$.
Let $\boldsymbol{\mathfrak{A}} = (\mathfrak{A}_1,\ldots,\mathfrak{A}_r) \in \left(\mathrm{Cl}_K^{\mathfrak{m}}\right)^r$ be a fixed $r$-tuple of ray classes modulo $\mathfrak{m}$. Define
\begin{equation*}
Q_K^{k,r}(X;\boldsymbol{\mathfrak{A}})
:=
\#\left\{
(\mathfrak{a}_1,\ldots,\mathfrak{a}_r)
\in
\prod_{i=1}^{r}
\mathcal{I}_K^{\mathfrak{m}}(X;\mathfrak{A}_i)
\suchthat
\text{there is no prime ideal $\mathfrak{p}$ with }
\mathfrak{p}^k\mid\mathfrak{a}_i
\text{ for every $i$}
\right\}.
\end{equation*}

Then, for any $X>1$,
\begin{align}\label{eqn:main-asymptotic-formula}
Q_K^{k,r}(X;\boldsymbol{\mathfrak{A}})
&=
\frac{\left(\rho_K^{\mathfrak{m}}\right)^r}
{\zeta_K(rk)}
\prod_{\mathfrak{p}\mid\mathfrak{m}_0}
\left(
1-\frac{1}{\Nrm{K/\Q}{\mathfrak{p}}^{rk}}
\right)^{-1}
X^r
+
E_{K,\mathfrak{m},k,r}(X),
\end{align}
where $\rho_K^{\mathfrak{m}}$ is given by
\eqref{eqn:residue}, and
\begin{equation}\label{eqn:main-error-term}
E_{K,\mathfrak{m},k,r}(X)
=
\begin{cases}
\displaystyle
O\left(X^{1/k}\right),
&
\text{if } k\left(r-\delta\right)<1,
\\[8pt]
\displaystyle
O\left(
X^{r-\delta}\log X
\right),
&
\text{if } k\left(r-\delta\right)=1,
\\[8pt]
\displaystyle
O\left(
X^{r-\delta}
\right),
&
\text{if } k\left(r-\delta\right)>1.
\end{cases}
\end{equation}
\end{theorem}

\begin{remark}
Observe that the main term in
\eqref{eqn:main-asymptotic-formula} is independent of the prescribed tuple of ray classes $\boldsymbol{\mathfrak{A}} \in \left(\Cl_K^{\mathfrak{m}}\right)^r$. Thus, relatively $k$-free $r$-tuples of integral ideals are
equidistributed among the
$\left(h_K^{\mathfrak{m}}\right)^r$ possible $r$-tuples of ray
classes modulo $\mathfrak{m}$, where
\[
h_K^{\mathfrak{m}}
:=
\#\Cl_K^{\mathfrak{m}}.
\]
In particular, when $r=1$, the $k$-free integral ideals are
equidistributed among the ray classes modulo $\mathfrak{m}$, so each
ray class contains asymptotically a proportion $1/h_K^{\mathfrak{m}}$ of all $k$-free ideals.
\end{remark}

\begin{remark}
In \cite{Sit10} and \cite{Sit20}, the authors count tuples of integral ideals of bounded norm without prescribing their ideal classes. Taking the trivial modulus $\mathfrak{m} = (1)$ in Theorem~\ref{thm:main-theorem} and summing over all tuples of ideal classes recovers the asymptotic formula of \cite{Sit10} and the specialization $\ell=r$ of the more general $\ell$-wise result of \cite{Sit20}, with improvements in certain cases. In particular, our result removes the logarithmic factor in the exceptional case $(r,k)=(2,1)$ and gives stronger error bounds for $k$-free ideals when $r=1$. See Remark~\ref{rem:sittinger-comparison} for a detailed comparison.
\end{remark}

We now record the $r = 1$ case of Theorem~\ref{thm:main-theorem} explicitly, together with the error terms obtained by taking the universally available admissible exponent $\delta = 1/d$.


\begin{theorem}\label{thm:k-free-ideals-in-ray-classes}
Let $K$ be a number field of degree $[K:\Q]=d$, let $\mathfrak{m}$
be a modulus for $K$, let $k\geq2$ be an integer, and let
$\mathfrak{A}\in\mathrm{Cl}_K^{\mathfrak{m}}$ be a fixed ray class
modulo $\mathfrak{m}$. Define
\begin{equation*}
Q_K^k(X;\mathfrak{A})
:=
\#\left\{
\mathfrak{a}\in
\mathcal{I}_K^{\mathfrak{m}}(X;\mathfrak{A})
\suchthat
\mathfrak{a}\text{ is $k$-free}
\right\}.
\end{equation*}
Then, for any $X>1$,
\begin{align}\label{eqn:k-free-ideals-ray-class-asymptotic}
Q_K^k(X;\mathfrak{A})
&=
\frac{\rho_K^{\mathfrak{m}}}{\zeta_K(k)}
\prod_{\mathfrak{p}\mid\mathfrak{m}_0}
\left(
1-\frac{1}{\Nrm{K/\Q}{\mathfrak{p}}^k}
\right)^{-1}
X
+
E_{K,\mathfrak{m},k}(X),
\end{align}
where $\rho_K^{\mathfrak{m}}$ is given by
\eqref{eqn:residue}.

For number fields of degree $d=2$, the error term satisfies
\[
E_{K,\mathfrak{m},k}(X)
=
\begin{cases}
O\left(X^{1/2}\log X\right),
&
\text{if } k=2,
\\[6pt]
O\left(X^{1/2}\right),
&
\text{if } k\geq3.
\end{cases}
\]

For number fields of degree $d\geq3$, we have
\[
E_{K,\mathfrak{m},k}(X)
=
O\left(X^{1-\frac1d}\right)
\qquad
\text{for every }k\geq2.
\]

Finally, for $K=\Q$, we have
\[
E_{\Q,\mathfrak{m},k}(X)
=
O\left(X^{1/k}\right).
\]
\end{theorem}

Theorem \ref{thm:k-free-ideals-in-ray-classes} recovers, as a special case, the classical formulas for the distribution of $k$-free integers \eqref{eqn:gegenbauer-formula} and for $k$-free integers in arithmetic progressions \eqref{eqn:cohen}. Indeed, when $K=\mathbb{Q}$ and $q \in \Z_{\geq 1}$, if we take the modulus $\mathfrak{m} = \mathfrak{m}_0 \mathfrak{m}_{\infty}$ with finite part $\mathfrak{m}_0=(q)$ and infinite part  $\mathfrak{m}_{\infty} = \infty$, the ray class group $\mathrm{Cl}_{\mathbb{Q}}^{\mathfrak{m}}$ is canonically isomorphic to $(\mathbb{Z}/q\mathbb{Z})^{\times}$ (see e.g. \cite[p. 127]{Lan94}, and for details also Example~\ref{ex:ray-class-groups-over-Q} below). Explicitly, if a fractional ideal of $\mathbb{Q}$ prime to $q$ is written as $(a/b)$ with $a, b \in \Z_{\geq 1}$, $\gcd(a, b) = 1$, and $\gcd(ab,q)=1$, then its ray class $[(a/b)]$ maps to
$$
[(a/b)]\longmapsto a b^{-1} \pmod q.
$$
In particular, the class of the integral ideal $(n)$ with $(n,q)=1$ maps to the reduced residue class of $n$ modulo $q$. Hence ray classes modulo $\mathfrak{m}$ recover precisely the reduced arithmetic progressions modulo $q$. Moreover, the $k$-free integral ideals of $\mathbb{Q}$ are exactly the ideals $(n)$ with $n$ a $k$-free positive integer. Therefore Theorem~\ref{thm:k-free-ideals-in-ray-classes} immediately yields the classical results from Gegenbauer and Cohen-Robinson as a corollary.

\begin{corollary}\label{cor:gegenbauer-cohen-robinson}
Let $k \ge 2$.
\begin{enumerate}
\item The number of $k$-free integers $n \le X$ satisfies
\[
\#\{  n \in \Z_{\ge 1} \suchthat \text{$n \le X$ and $n$ is $k$-free} \}
= \frac{1}{\zeta(k)} \cdot X + O(X^{1/k}).
\]

\item Let $q \ge 1$ and $(a,q)=1$. Then

\begin{equation*}
\#\{ n \in \Z_{\ge 1} \suchthat \text{$n \le X, \, n \equiv a \bmod q$ and $n$ is $k$-free} \}
= \frac{1}{q\,\zeta(k)}
  \prod_{p\mid q}\left(1 - \tfrac{1}{p^k}\right)^{-1}
  \cdot X
  + O(X^{1/k})
\end{equation*}
\end{enumerate}
\end{corollary}

For the ordinary ideal class group $\Cl_K$ and the
narrow class group $\Cl_K^+$, we may instead use the stronger admissible exponent $\delta = 2/(d+1)$. These groups correspond, respectively, to the trivial
modulus $\mathfrak{m}=(1)$ and to the modulus
$\mathfrak{m}=\mathfrak{m}_{\infty}^{\mathrm{full}}$ whose finite part is trivial and whose infinite part consists of all real places of $K$ (see equation \eqref{eqn:narrow-class-group-modulus}). Thus, the case $r=1$ of Theorem~\ref{thm:main-theorem}, with $\delta=2/(d+1)$, gives the following improved estimates.

\begin{theorem}\label{thm:k-free-ideals-class-and-narrow-class-groups}
Let $K$ be a number field of degree $d\geq2$, and let $k\geq2$ be an
integer.

If $\mathfrak{C}\in\Cl_K$ is an ideal class, define
\[
Q_K^k(X;\mathfrak{C})
:=
\#\left\{
\mathfrak{a}\trianglelefteq\mathcal{O}_K
\suchthat
\Nrm{K/\Q}{\mathfrak{a}}\leq X,\,
\mathfrak{a}\in\mathfrak{C},\,
\mathfrak{a}\text{ is $k$-free}
\right\}.
\]
Then
\[
Q_K^k(X;\mathfrak{C})
=
\frac{\rho_K}{h_K\zeta_K(k)}X
+
E_{K,k}(X;\mathfrak{C}).
\]

Similarly, if $\mathfrak{C}^+\in\Cl_K^+$ is a narrow ideal class,
define
\[
Q_K^k(X;\mathfrak{C}^+)
:=
\#\left\{
\mathfrak{a}\trianglelefteq\mathcal{O}_K
\suchthat
\Nrm{K/\Q}{\mathfrak{a}}\leq X,\,
\mathfrak{a}\in\mathfrak{C}^+,\,
\mathfrak{a}\text{ is $k$-free}
\right\}.
\]
Then
\[
Q_K^k(X;\mathfrak{C}^+)
=
\frac{\rho_K}{h_K^+\zeta_K(k)}X
+
E_{K,k}^+(X;\mathfrak{C}^+).
\]
In what follows, $E(X)$ denotes either $E_{K,k}(X;\mathfrak{C})$ or
$E_{K,k}^+(X;\mathfrak{C}^+)$, as appropriate. Then $E(X)$ satisfies the following bounds. 

\noindent For number fields of degree $d=2$,
\[
E(X)
=
\begin{cases}
O\left(X^{1/2}\right),
& \text{if } k=2,\\[6pt]
O\left(X^{1/3}\log X\right),
& \text{if } k=3,\\[6pt]
O\left(X^{1/3}\right),
& \text{if } k\geq4.
\end{cases}
\]
For number fields of degree $d=3$,
\[
E(X)
=
\begin{cases}
O\left(X^{1/2}\log X\right),
& \text{if } k=2,\\[6pt]
O\left(X^{1/2}\right),
& \text{if } k\geq3.
\end{cases}
\]
Finally, for number fields of degree $d\geq4$,
\[
E(X)
=
O\left(X^{1-\frac{2}{d+1}}\right)
\qquad
\text{for every }k\geq2.
\]
Here $h_K$ and $h_K^+$ denote the class number and narrow class number
of $K$, respectively, and
\[
\rho_K
=
\underset{s=1}{\operatorname{Res}}\,\zeta_K(s)
\]
is the residue of the Dedekind zeta function $\zeta_K(s)$ at $s=1$.
\end{theorem}

Summing the first asymptotic formula in Theorem~\ref{thm:k-free-ideals-class-and-narrow-class-groups} over all ideal
classes gives the number field analogue of Gegenbauer's formula \eqref{eqn:gegenbauer-formula}.

\begin{corollary}\label{cor:gegenbauer-number-field}
Let $K$ be a number field of degree $d\geq2$, let $k\geq2$ be an integer, and let
\[
Q_K^k(X)
:=
\#\left\{
\mathfrak{a}\trianglelefteq\mathcal{O}_K
\suchthat
\Nrm{K/\Q}{\mathfrak{a}}\leq X
\text{ and }
\mathfrak{a}\text{ is $k$-free}
\right\}
\]
denote the number of $k$-free integral ideals of $K$ of norm at most $X$. Then,
for any $X\geq2$,
\[
Q_K^k(X)
=
\frac{\rho_K}{\zeta_K(k)}X
+
E_{K,k}(X),
\]
where
\[
E_{K,k}(X)
=
\begin{cases}
O\left(X^{1/2}\right), & \text{if } d=2 \text{ and } k=2,\\[4pt]
O\left(X^{1/3}\log X\right), & \text{if } d=2 \text{ and } k=3,\\[4pt]
O\left(X^{1/3}\right), & \text{if } d=2 \text{ and } k\geq4,\\[4pt]
O\left(X^{1/2}\log X\right), & \text{if } d=3 \text{ and } k=2,\\[4pt]
O\left(X^{1/2}\right), & \text{if } d=3 \text{ and } k\geq3,\\[4pt]
O\left(X^{1-\frac{2}{d+1}}\right), & \text{if } d\geq4.
\end{cases}
\]
\end{corollary}


\subsection{Computational examples}\label{section:computational-examples}

In this subsection, we present explicit computational examples illustrating the asymptotic formulas in Theorems \ref{thm:k-free-ideals-in-ray-classes} and \ref{thm:k-free-ideals-class-and-narrow-class-groups}. We implemented functions in \texttt{SageMath} \cite{SageMathCoCalc} to compute directly the values of the counting function $Q_K^{k}(X;\mathfrak{A})$ introduced in Theorem~\ref{thm:main-theorem}, namely the number of $k$-free integral ideals of $\mathcal{O}_K$ with norm at most $X$ lying in a fixed ray class $\mathfrak{A}$. We also implemented a function that computes the corresponding main term in the asymptotic formula \eqref{eqn:main-asymptotic-formula}. The code used for these computations is available in the \texttt{GitHub} repository \cite{k-free-ideals-code}.

We start by giving an example of Theorem \ref{thm:k-free-ideals-class-and-narrow-class-groups}.

\begin{example}
Consider the cubic field $K=\Q(\sqrt[3]{7})$. The class number of $K$ is $h_K=3$. Thus there are three ideal classes in the class group  $\Cl_K$ and they are given by
\[
\mathfrak{A}_1=[\mathcal{O}_K],
\qquad
\mathfrak{A}_2=[\langle2,\sqrt[3]{7}+1 \rangle],
\qquad
\mathfrak{A}_3=[\langle 3,\sqrt[3]{7}-1 \rangle].
\]
Hence, for each $i=1,2,3$, the quantity $Q_K^{k}(X;\mathfrak{A}_i)$
counts the number of $k$-free integral ideals $\mathfrak{a}\trianglelefteq \mathcal{O}_K$ with $\Nrm{K/\Q}{\mathfrak{a}}\leq X$ and whose ideal class is $\mathfrak{A}_i$.
In this example the field has signature $(r_1, r_2) = (1, 1)$, and discriminant $d_K=-1323$. The unit rank of $K$ is $r_K = r_1 + r_2 -1 = 1$, and a fundamental unit of $K=\Q(\sqrt[3]{7})$ is $\varepsilon=2-\sqrt[3]{7}$.

Hence
\[
\operatorname{Reg}_K
=
|\log|\varepsilon||
=
-\log(2-\sqrt[3]{7})
=
\log\left(\frac{1}{2-\sqrt[3]{7}}\right).
\]
Equivalently, since
\[
\frac{1}{2-\sqrt[3]{7}}
=
4+2\sqrt[3]{7}+\sqrt[3]{49},
\]
we have
\[
\operatorname{Reg}_K
=
\log\left(4+2\sqrt[3]{7}+\sqrt[3]{49}\right).
\]
Therefore, using $r_1=1$, $r_2=1$, $w_K=2$, $h_K=3$, and $d_K=-1323$, the analytic class number formula gives
\[
\frac{\rho_K}{h_K}
=
\frac{2^{r_1}(2\pi)^{r_2}\operatorname{Reg}_K}
{w_K |d_K|^{1/2}}
=
\frac{2\pi}{\sqrt{1323}}
\log\left(4+2\sqrt[3]{7}+\sqrt[3]{49}\right).
\]
Hence, for every ideal class $\mathfrak{C}\in \Cl_K$ the asymptotic formula becomes
\[
Q_K^k(X;\mathfrak{C})
=
\frac{2\pi}{\sqrt{1323}\,\zeta_K(k)}
\log\left(4+2\sqrt[3]{7}+\sqrt[3]{49}\right)X
+
E_{K,k}(X;\mathfrak{C}).
\]
Moreover, the error term satisfies
\[
E_{K,k}(X;\mathfrak{C}) = 
\begin{cases}
O\left(X^{1/2}\log X\right),
& \text{if } k=2,\\[6pt]
O\left(X^{1/2}\right),
& \text{if } k\geq3.
\end{cases}
\]

In Table \ref{table:cubic-field-class-group-squarefree} we compare the exact counts in each of the three ideal classes with the common main term.

\begin{table}[H]
\centering
\renewcommand{\arraystretch}{1.3}
\begin{tabular}{c c c c c}
\toprule
$X$
&
$Q_K^{2}(X;\mathfrak{A}_1)$
&
$Q_K^{2}(X;\mathfrak{A}_2)$
&
$Q_K^{2}(X;\mathfrak{A}_3)$
&
$\dfrac{\rho_K}{3}\dfrac{X}{\zeta_K(2)}$
\\
\midrule
10000  & 2382  & 2423  & 2413  & 2407.27 \\
20000  & 4825  & 4809  & 4832  & 4814.54 \\
30000  & 7242  & 7194  & 7214  & 7221.80 \\
40000  & 9657  & 9606  & 9626  & 9629.07 \\
50000  & 12051 & 12037 & 12016 & 12036.34 \\
60000  & 14475 & 14434 & 14459 & 14443.61 \\
70000  & 16854 & 16879 & 16884 & 16850.87 \\
80000  & 19267 & 19259 & 19225 & 19258.14 \\
90000  & 21677 & 21710 & 21670 & 21665.41 \\
100000 & 24064 & 24111 & 24090 & 24072.68 \\
\bottomrule\\
\end{tabular}
\caption{Distribution of squarefree ideal counts across the three ideal classes in $\Cl_K$ for $K=\Q(\sqrt[3]{7})$ and comparison with the common main term in the asymptotic formula from Theorem \ref{thm:k-free-ideals-class-and-narrow-class-groups}.}
\label{table:cubic-field-class-group-squarefree}
\end{table}

\end{example}

We now give an example of Theorem \ref{thm:k-free-ideals-in-ray-classes} with a nontrivial modulus.

\begin{example}\label{example:nontrivial-ray-class-example}
Consider the real quadratic field $K=\Q(\sqrt{5})$. We take the modulus
$\mathfrak{m}=\mathfrak{m}_0\mathfrak{m}_{\infty}$, where
$\mathfrak{m}_0=\langle 4\rangle=\langle 2\rangle^2$ and
$\mathfrak{m}_{\infty}=\infty_1\infty_2$, so that the infinite part contains both infinite places of $K$.
By the ray class number formula \eqref{eqn:ray-class-number-formula}, we have
\[
h_K^{\mathfrak{m}}
=
\frac{\varphi(\mathfrak{m}_0)2^{\#\mathfrak{m}_{\infty}}h_K}
{\left[\mathcal{O}_K^{\times}:\mathcal{O}_K^{\times}\cap K^{\mathfrak{m},1}\right]}.
\]
In this case, $h_K=1$, $\varphi(\mathfrak{m}_0)=12$, $\#\mathfrak{m}_{\infty}=2$, and
\[
\left[\mathcal{O}_K^{\times}:\mathcal{O}_K^{\times}\cap K^{\mathfrak{m},1}\right]=12.
\]
Therefore $h_K^{\mathfrak{m}}=4$, in agreement with the direct computation using our \texttt{SageMath} implementation.

Thus the ray class group $\Cl_K^{\mathfrak{m}}$ has four elements. Representatives for these ray classes are given by
\[
\mathfrak{A}_1=[\mathcal{O}_K],
\quad
\mathfrak{A}_2=
\left[
\left\langle
\frac{1+3\sqrt{5}}{2}
\right\rangle
\right],
\quad
\mathfrak{A}_3=
[\langle -\sqrt{5}\rangle],
\quad
\mathfrak{A}_4=
\left[
\left\langle
\frac{-1+3\sqrt{5}}{2}
\right\rangle
\right].
\]

We next compute the residue
\[
\rho_K^{\mathfrak{m}}
=
\operatorname*{Res}_{s=1}\zeta_K(s;\mathfrak{A}_i)
\]
appearing in the main term of equation \eqref{eqn:main-asymptotic-formula}. Using \eqref{eqn:residue}, we obtain
\[
\rho_K^{\mathfrak{m}}
=
\frac{
2^{r_1}(2\pi)^{r_2}\operatorname{Reg}_{\mathfrak{m}}
}{
w_{\mathfrak{m}}|d_K|^{1/2}2^{\#\mathfrak{m}_{\infty}}\Nrm{K/\Q}{\mathfrak{m}_0}
}
=
\frac{3}{8\sqrt{5}}
\log\left(\frac{1+\sqrt{5}}{2}\right).
\]
Here we have used that $r_1=2$, $r_2=0$, $d_K=5$,
$\Nrm{K/\Q}{\mathfrak{m}_0}=16$, and
\[
\operatorname{Reg}_{\mathfrak{m}}
=
\frac{w_{\mathfrak{m}}}{w_K}
\left[\mathcal{O}_K^{\times}:\mathcal{O}_K^{\times}\cap K^{\mathfrak{m},1}\right]
\operatorname{Reg}_K.
\]
In the present example, $w_{\mathfrak{m}}=1$, $w_K=2$, and
\[
\operatorname{Reg}_K
=
\log\left(\frac{1+\sqrt{5}}{2}\right).
\]

Since the only prime ideal dividing $\mathfrak{m}_0$ is $\langle 2\rangle$, and since
$\Nrm{K/\Q}{\langle 2\rangle}=4$, we have
\[
\prod_{\mathfrak{p}\mid \mathfrak{m}_0}
\left(
1-\frac{1}{\Nrm{K/\Q}{\mathfrak{p}}^k}
\right)^{-1}
=
\left(1-\frac{1}{4^k}\right)^{-1}.
\]
Therefore, for each ray class $\mathfrak{A}_i$, the asymptotic formula
\eqref{eqn:main-asymptotic-formula} becomes
\[
Q_K^k(X;\mathfrak{A}_i)
=
\frac{3}{8\sqrt{5}}
\log\left(\frac{1+\sqrt{5}}{2}\right)
\left(1-\frac{1}{4^k}\right)^{-1}
\frac{X}{\zeta_K(k)}
+
E_{K,\mathfrak{m},k}(X).
\]
Since $K$ is quadratic, the error term is
\[
E_{K,\mathfrak{m},k}(X)
=
\begin{cases}
O\left(X^{1/2}\log X\right),
&
\text{if } k=2,
\\[6pt]
O\left(X^{1/2}\right),
&
\text{if } k\geq3.
\end{cases}
\]

In Tables~\ref{table:nontrivial-ray-class-example1-squarefree},
\ref{table:nontrivial-ray-class-example1-cubefree}, and
\ref{table:nontrivial-ray-class-example1-fourthfree}, we present computations for
$k=2,3,4$. We also plot these functions in Figure~\ref{fig:plots}. For convenience, we write
\[
C_K(\mathfrak{m},k)
:=
\frac{3}{8\sqrt{5}}
\log\left(\frac{1+\sqrt{5}}{2}\right)
\left(1-\frac{1}{4^k}\right)^{-1},
\]
so that the predicted main term for each ray class is
\[
C_K(\mathfrak{m},k)\frac{X}{\zeta_K(k)}.
\]

\begin{table}[h!]
\centering
\renewcommand{\arraystretch}{1.4}
\begin{tabular}{c c c c c c}
\toprule
$X$ & $Q_K^{2}(X; \mathfrak{A}_1)$ & $Q_K^{2}(X; \mathfrak{A}_2)$ & $Q_K^{2}(X; \mathfrak{A}_3)$ & $Q_K^{2}(X; \mathfrak{A}_4)$ & $C_K(\mathfrak{m}, 2) \dfrac{X}{\zeta_K(2)}$ \\
\midrule
100   & 6   & 9   & 7   & 9   & 7.41 \\
500   & 32  & 40  & 37  & 40  & 37.05 \\
1000  & 66  & 78  & 74  & 78  & 74.10 \\
5000  & 362 & 374 & 371 & 374 & 370.51 \\
10000 & 739 & 748 & 733 & 748 & 741.02 \\
\bottomrule\\
\end{tabular}
\caption{Distribution of counts across ray classes for $k=2$ and comparison with the main term.}
\label{table:nontrivial-ray-class-example1-squarefree}
\end{table}

\begin{table}[h!]
\centering
\renewcommand{\arraystretch}{1.4}
\begin{tabular}{c c c c c c}
\toprule
$X$ & $Q_K^{3}(X; \mathfrak{A}_1)$ & $Q_K^{3}(X; \mathfrak{A}_2)$ & $Q_K^{3}(X; \mathfrak{A}_3)$ & $Q_K^{3}(X; \mathfrak{A}_4)$ & $C_K(\mathfrak{m}, 3) \dfrac{X}{\zeta_K(3)}$ \\
\midrule
100   & 8   & 9   & 7   & 9   & 7.98 \\
500   & 38  & 42  & 39  & 42  & 39.89 \\
1000  & 78  & 82  & 78  & 82  & 79.78 \\
5000  & 397 & 400 & 404 & 400 & 398.92 \\
10000 & 801 & 798 & 796 & 798 & 797.85 \\
\bottomrule\\
\end{tabular}
\caption{Distribution of counts across ray classes for $k=3$ and comparison with the main term.}
\label{table:nontrivial-ray-class-example1-cubefree}
\end{table}

\begin{table}[h!]
\centering
\renewcommand{\arraystretch}{1.4}
\begin{tabular}{c c c c c c}
\toprule
$X$ & $Q_K^{4}(X; \mathfrak{A}_1)$ & $Q_K^{4}(X; \mathfrak{A}_2)$ & $Q_K^{4}(X; \mathfrak{A}_3)$ & $Q_K^{4}(X; \mathfrak{A}_4)$ & $C_K(\mathfrak{m}, 4) \dfrac{X}{\zeta_K(4)}$ \\
\midrule
100   & 8   & 9   & 7   & 9   & 8.05 \\
500   & 38  & 42  & 40  & 42  & 40.27 \\
1000  & 78  & 82  & 80  & 82  & 80.55 \\
5000  & 399 & 404 & 408 & 404 & 402.74 \\
10000 & 808 & 808 & 800 & 808 & 805.48 \\
\bottomrule\\
\end{tabular}
\caption{Distribution of counts across ray classes for $k=4$ and comparison with the main term.}
\label{table:nontrivial-ray-class-example1-fourthfree}
\end{table}

Interestingly, as Tables \ref{table:nontrivial-ray-class-example1-squarefree}, \ref{table:nontrivial-ray-class-example1-cubefree} and \ref{table:nontrivial-ray-class-example1-fourthfree} show, the quantities $Q_K^{k}(X; \mathfrak{A}_2)$ and $Q_K^{k}(X; \mathfrak{A}_4)$ give the exact same value for every $X>1$ and every $k \geq 2$. This is not an accident. For an explanation of this see Remark \ref{remark:same-counting-values}.

\begin{figure}[h!]
\centering
\begin{tikzpicture}
\begin{axis}[
    width=13cm,
    height=8cm,
    xlabel={$X$},
    ylabel={Counting functions},
    grid=both,
    legend style={
        at={(0.02,0.98)},
        anchor=north west,
        draw=none,
        fill=white,
        fill opacity=0.8,
        text opacity=1
    },
]

\addplot[
    only marks,
    color=blue,
    mark=*,
    mark size=2pt
] coordinates {
(100,6)(200,11)(300,19)(400,27)(500,32)(600,44)(700,47)(800,57)(900,62)(1000,66)
(1100,77)(1200,81)(1300,89)(1400,100)(1500,106)(1600,112)(1700,121)(1800,131)(1900,138)(2000,142)
};
\addlegendentry{$Q_K^{2}(X; \mathfrak{A}_1)$}

\addplot[
    only marks,
    color=red,
    mark=square*,
    mark size=2pt
] coordinates {
(100,9)(200,17)(300,23)(400,31)(500,40)(600,46)(700,55)(800,62)(900,69)(1000,78)
(1100,84)(1200,94)(1300,102)(1400,107)(1500,112)(1600,121)(1700,129)(1800,135)(1900,144)(2000,152)
};
\addlegendentry{$Q_K^{2}(X; \mathfrak{A}_2)$}

\addplot[
    only marks,
    color=teal,
    mark=triangle*,
    mark size=2pt
] coordinates {
(100,7)(200,14)(300,24)(400,29)(500,37)(600,44)(700,52)(800,56)(900,67)(1000,74)
(1100,78)(1200,90)(1300,92)(1400,100)(1500,111)(1600,116)(1700,122)(1800,130)(1900,141)(2000,147)
};
\addlegendentry{$Q_K^{2}(X; \mathfrak{A}_3)$}

\addplot[
    only marks,
    color=orange,
    mark=diamond*,
    mark size=2pt
] coordinates {
(100,9)(200,17)(300,23)(400,31)(500,40)(600,46)(700,55)(800,62)(900,69)(1000,78)
(1100,84)(1200,94)(1300,102)(1400,107)(1500,112)(1600,121)(1700,129)(1800,135)(1900,144)(2000,152)
};
\addlegendentry{$Q_K^{2}(X; \mathfrak{A}_4)$}

\addplot[
    color=black,
    thick
] coordinates {
(100,7.410168085)(200,14.82033617)(300,22.23050426)(400,29.64067234)(500,37.05084043)
(600,44.46100851)(700,51.87117660)(800,59.28134468)(900,66.69151277)(1000,74.10168085)
(1100,81.51184894)(1200,88.92201702)(1300,96.33218511)(1400,103.7423532)(1500,111.1525213)
(1600,118.5626894)(1700,125.9728574)(1800,133.3830255)(1900,140.7931936)(2000,148.2033617)
};
\addlegendentry{$C_K(\mathfrak{m},2)\dfrac{X}{\zeta_K(2)}$}

\end{axis}
\end{tikzpicture}
\caption{A plot of the values of the four counting functions $Q_K^{2}(X; \mathfrak{A}_i)$ and of the main term in the asymptotic for $K = \Q(\sqrt{5})$ and modulus $\mathfrak{m} = \langle 2 \rangle^2 \infty_1 \infty_2$. This gives a visual representation of the equidistribution of ideals in the four different ray classes.}
\label{fig:plots}
\end{figure}
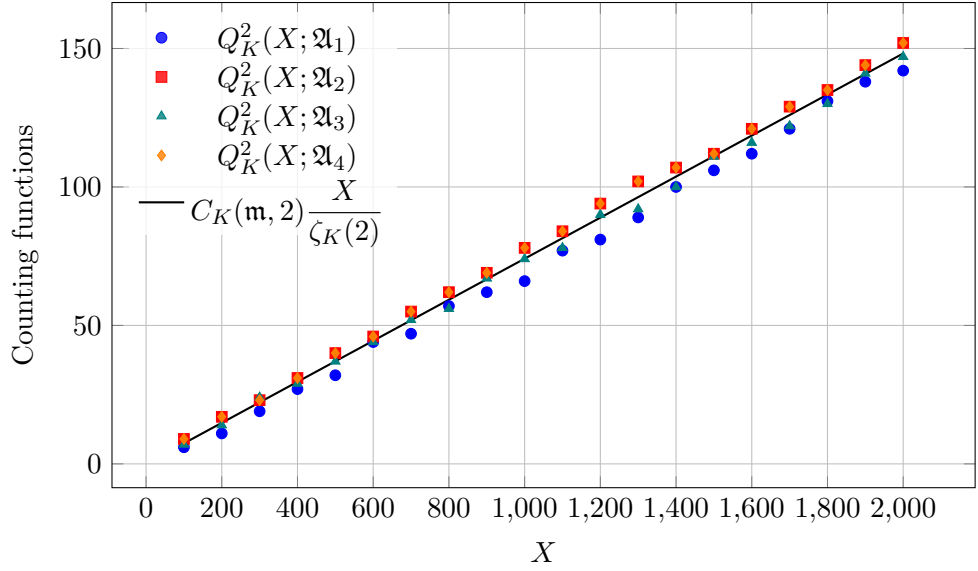

\end{example}

\begin{remark}\label{remark:same-counting-values}
The equality $Q_K^k(X;\mathfrak A_2)=Q_K^k(X;\mathfrak A_4)$ observed in Tables \ref{table:nontrivial-ray-class-example1-squarefree}, \ref{table:nontrivial-ray-class-example1-cubefree} and \ref{table:nontrivial-ray-class-example1-fourthfree} can be explained by a symmetry coming from the Galois structure of the field. Let $\sigma \colon K \longrightarrow K$ denote the nontrivial automorphism of $K=\mathbb{Q}(\sqrt{5})$. Since the modulus $\mathfrak m= \langle 4 \rangle \infty_1\infty_2$ is stable under $\sigma$ (the finite part $\langle 4 \rangle$ is fixed and the infinite places are permuted), conjugation by $\sigma$ induces an action on the ray class group $\mathrm{Cl}_K^{\mathfrak m}$.

More concretely, if $\mathfrak a \subseteq \mathcal{O}_K$ is an integral ideal coprime to $\mathfrak m_0$, then $\sigma(\mathfrak a)$ is again such an ideal, and this operation respects the defining congruence and positivity conditions of the ray class group. Thus $\sigma$ permutes the ray classes. In Example \ref{example:nontrivial-ray-class-example}, one easily verifies that $\sigma(\mathfrak A_2)=\mathfrak A_4$.

This symmetry has a direct consequence for the counting problem, which is that the map $\mathfrak a \mapsto \sigma(\mathfrak a)$ defines a bijection between the sets of integral ideals in the ray classes $\mathfrak A_2$ and $\mathfrak A_4$. Moreover, it preserves the norm and the property of being $k$-free, since it simply permutes the prime ideal factorization. Therefore the corresponding counting functions agree identically, that is
\[
Q_K^k(X;\mathfrak A_2)=Q_K^k(X;\mathfrak A_4)
\quad \text{for all } X>0 \text{ and } k \geq 2.
\]

More generally, if $K/\Q$ is a Galois extension, then whenever an automorphism of the Galois group $\Gal{(K/\Q)}$ stabilizes the modulus $\mathfrak m$, it induces a permutation of the ray class group $\op{Cl}_K^{\mathfrak{m}}$, and any two ray classes in the same orbit under this action give rise to identical counting functions.
\end{remark}

\subsection{Acknowledgments}

We would like to thank the organizers of Rethinking Number Theory (RNT), an American Institute of Mathematics research community. RNT provides an accessible, collaborative, and enjoyable research environment for interested mathematicians at any stage of their career. RNT is also supported by the National Science Foundation (DMS--2418528, DMS--2201085). We are especially grateful to the current organizers of RNT, Jen Berg, Heidi Goodson, and Allechar Serrano López. This work began during RNT 6 as part of the Cornflower Blue research team, and it would not have been possible without the workshop and the efforts of its organizers. In particular, A. Barquero-Sanchez and N. Sirolli are grateful for the wonderful opportunity to serve as co-leaders of the research project that led to this paper.

A. Barquero-Sanchez gratefully acknowledges the Centro de Investigación en Matemática Pura y Aplicada (CIMPA) and the School of Mathematics of the University of Costa Rica for their administrative assistance and support throughout this project. This work was carried out as part of the research project 821-C5-139, led by A. Barquero-Sanchez and registered with the Vicerrectoría de Investigación of the University of Costa Rica.

C. Spivey would like to thank her advisor, Keith Conrad, for his support, comments, and suggestions during this project.

We used ChatGPT to assist in developing SageMath wrappers for selected PARI/GP functionality related to ray class groups in the accompanying computational code. We also used Claude to assist with proofreading and language editing, including improving the readability and flow of the exposition in parts of the manuscript, and to independently verify the numerical computations of Section~\ref{section:computational-examples}. All generated code, computations, and suggestions were reviewed and verified by the authors, who are solely responsible for the contents of this paper.

\section{Ray class groups, ray class zeta functions and ideal counting functions}\label{section:ray-class-groups}

In this section we recall the definition and some basic facts about the ray class group associated to a given modulus in a number field $K$ and then introduce the main counting functions that we study in this paper. We closely follow the notation of Sutherland's lecture notes \cite{Sut21}, as it is more modern and explicit than that used in Lang's book and other classical references. Proofs of the facts cited here can be found in \cite[Chaps. VI and VIII]{Lan94} and in \cite{Sut21}.

\subsection{Ray class groups and arithmetic progressions}

Ray class groups provide the natural number-field analogue of reduced residue classes modulo an integer. We first recall their definition and then make this analogy explicit in the case $K=\mathbb{Q}$ in Example \ref{ex:ray-class-groups-over-Q}.

We now recall the definition of ray class groups. We fix throughout this section a number field $K$ of degree $d$, and we denote its ring of integers by $\mathcal{O}_K$. We recall that the normalized absolute values $v \colon K \longrightarrow [0, \infty)$ are given as follows. The \textit{archimedean absolute values} are defined in terms of the embeddings $K \hooklongrightarrow \C$. Specifically, if $\sigma \colon K \hooklongrightarrow \R$ is a real embedding, then the corresponding absolute value is defined by $v_{\sigma}(x) := |\sigma(x)|$ for every $x \in K$, whereas if $\rho \colon K \hooklongrightarrow \C$ is a non-real complex embedding, the corresponding archimedean absolute value is given by $v_{\rho}(x) := |\rho(x)|^2$ for every $x \in K$, where the vertical bars $|\cdot|$ mean the absolute value in $\R$ or the modulus in $\C$, respectively.

On the other hand, the \textit{nonarchimedean absolute values} are defined in terms of the nonzero prime ideals of $K$. Thus, let $\mathfrak{p} \trianglelefteq \mathcal{O}_K$ be a nonzero prime ideal of $K$. Then the corresponding normalized absolute value is defined by 
$$
v_{\mathfrak{p}}(x) := \mathbf{N}_{K/\Q}(\mathfrak{p})^{-\operatorname{ord}_{\mathfrak{p}}(x)}
$$
for every $x \in K^{\times}$ and $v_{\mathfrak{p}}(0) := 0$. 

We let $M_K$ be the set of all the different normalized absolute values defined above. Moreover, we also let $M_K^{\infty}$ (resp. $M_K^0$) be the subset of archimedean (resp. nonarchimedean) absolute values in $M_K$. Note in particular that if $\rho \colon K \hooklongrightarrow \C$ is a complex embedding, then the absolute values $v_{\rho}$ and $v_{\overline{\rho}}$ are actually the same function, so each pair of complex conjugate embeddings $\rho, \overline{\rho}$ accounts for exactly one archimedean absolute value in $M_K$.

Now, a \textit{modulus} $\mathfrak{m}$ for $K$ is a function $\mathfrak{m} \colon M_K \longrightarrow \Z_{\geq 0}$ such that
\begin{enumerate}
    \item $\mathfrak{m}(v) = 0$ for almost every $v \in M_K$,
    \item $\mathfrak{m}(v_{\rho}) = 0$ for every complex embedding $\rho \colon K \hooklongrightarrow \C$, and
    \item $\mathfrak{m}(v_{\sigma}) \in \{ 0, 1 \}$ for every real embedding $\sigma \colon K \hooklongrightarrow \R$.
\end{enumerate}
Alternatively, such a modulus is also commonly written as formal product $\mathfrak{m} = \mathfrak{m}_0 \mathfrak{m}_\infty$, where
\begin{align*}
\mathfrak{m}_0:=\prod_{v_{\mathfrak{p}} \in M_K^{0}} \mathfrak{p}^{\mathfrak{m}(v_{\mathfrak{p}})}
\quad \text{and} \quad
\mathfrak{m}_{\infty}:=\prod_{\substack{v \in M_K^{\infty}\\\text{$v$ real}}} v^{\mathfrak{m}(v)}
.
\end{align*}
Here $\mathfrak{m}_0$ is called the \textit{finite part} of $\mathfrak{m}$ and is regarded as a nonzero integral ideal $\mathfrak{m}_0 \trianglelefteq \mathcal{O}_K$, while $\mathfrak{m}_{\infty}$ is called the \textit{infinite part} of $\mathfrak{m}$. Thus, the infinite part of a modulus may be viewed
simply as a subset of the real places of $K$.

Now, associated to a modulus $\mathfrak{m}$ we introduce the following notation. First we let $\mathcal{F}_K$ be the multiplicative group of fractional ideals of $K$ and $\mathcal{F}_K^{\mathfrak{m}}$ be the subgroup of $\mathcal{F}_K$ of fractional ideals coprime to $\mathfrak{m}_0$, defined by
\begin{align*}
\mathcal{F}_K^{\mathfrak{m}} := \left\{
\mathfrak{f}\in \mathcal{F}_K
\ \middle|\
\operatorname{ord}_{\mathfrak{p}}(\mathfrak{f})=0
\text{ for every } \mathfrak{p}\mid \mathfrak{m}_0
\right\}.
\end{align*}

We define $K^{\mathfrak{m},1}$ to be the subgroup of $K^{\times}$ consisting of elements $\alpha \in K^{\times}$ that satisfy the conditions
\begin{enumerate}
\item $\gcd{(\langle \alpha \rangle, \mathfrak{m}_0)} = 1$,
\item $\ord_\mathfrak{p}(\alpha - 1) \geq \mathfrak{m}(v_{\mathfrak{p}})$ for every $\mathfrak{p} \mid \mathfrak{m}_0$, and
\item $\sigma(\alpha) > 0$ for every $v_{\sigma} \in M_K^{\infty}$ for which $\mathfrak{m}(v_{\sigma}) = 1$.
\end{enumerate}
We also define the subgroup of \emph{rays} $\mathcal{R}_K^{\mathfrak{m}} \subseteq \mathcal{F}_K^{\mathfrak{m}}$ consisting of the principal fractional ideals $\langle \alpha \rangle \in \mathcal{F}_K^{\mathfrak{m}}$ such that the generator $\alpha \in K^{\mathfrak{m}, 1}$.

Then, the \textit{ray class group} of $K$ modulo $\mathfrak{m}$ is defined to be the quotient
\begin{align*}
\mathrm{Cl}_K^{\mathfrak{m}}:=\mathcal{F}_K^{\mathfrak{m}} / \mathcal{R}_K^{\mathfrak{m}}.
\end{align*}
This group is known to be finite and its cardinality $h_K^{\mathfrak{m}}:=\# \mathrm{Cl}_K^{\mathfrak{m}}$ is given by 

\begin{align}\label{eqn:ray-class-number-formula}
h_K^{\mathfrak{m}} = \frac{\varphi(\mathfrak{m}_0) 2^{\#\mathfrak{m}_{\infty}} h_K}{\left[\mathcal{O}_K^{\times}: \mathcal{O}_K^{\times} \cap K^{\mathfrak{m}, 1}\right]},
\end{align}
where $h_K:=\# \mathrm{Cl}_K$ is the class number of $K$ and
$$
\varphi(\mathfrak{m}_0) := \#\left(\mathcal{O}_K / \mathfrak{m}_0\right)^{\times}= \Nrm{K/\Q}{\mathfrak{m}_0} \prod_{\mathfrak{p} \mid \mathfrak{m}_0}\left( 1 - \frac{1}{\Nrm{K/\Q}{\mathfrak{p}}} \right).
$$
For a proof see for example \cite[Chap. VI, \S 1, Thm. 1]{Lan94} or \cite[Corollary 21.9]{Sut21}. Moreover, the class number $h_K$ divides the ray class number $h_K^{\mathfrak{m}}$.

\begin{remark}
Observe that if we let $\mathbf{0} \colon M_K \longrightarrow \Z_{\geq 0}$ denote the zero modulus with $\mathbf{0}(v) = 0$ for every $v \in M_K$, then we have $\mathrm{Cl}_K^{\mathbf{0}} = \mathrm{Cl}_K$. Note that with the convention that sees a modulus as a formal product, this corresponds to taking the modulus to be $\mathfrak{m} = (1)$, just the trivial ideal without including any infinite places.
\end{remark}

As was mentioned in the introduction, ray class groups provide a natural way to generalize arithmetic progressions to arbitrary algebraic number fields. The following classic example makes this observation precise.

\begin{example}\label{ex:ray-class-groups-over-Q}
Let $q \in \Z_{\geq 1}$ and consider the modulus for $\Q$ given by $\mathfrak{m} = (q) \cdot \infty$, with finite part $\mathfrak{m}_0=(q)$ and infinite part  $\mathfrak{m}_{\infty} = \infty$ corresponding to the unique real place of $\Q$. In this case the ray class group modulo $\mathfrak{m}$ is given by the quotient
\[
\Cl_{\Q}^{\mathfrak{m}} = \mathcal{F}_{\Q}^{\mathfrak{m}}/\mathcal{R}_{\Q}^{\mathfrak{m}},
\]
where $\mathcal{F}_{\Q}^{\mathfrak{m}}$ is the group of fractional ideals of $\Z$ that are relatively prime to $(q)$ and $\mathcal{R}_{\Q}^{\mathfrak{m}}$ is the subgroup generated by principal ideals $(\alpha)$ with $\alpha>0$ and
\[
\ord_p(\alpha-1)\geq \ord_p(q)
\qquad\text{for every prime }p\mid q.
\]
Since $\Z$ is a PID, every fractional ideal in $\mathcal{F}_{\Q}^{\mathfrak{m}}$ has a unique positive generator. Thus every such ideal can be written uniquely as $(a/b)$ with $a, b \in \Z_{\geq 1}$, $\gcd{(a, b)} = 1$ and $\gcd{(ab, q)} = 1$. Hence $a$ and $b$ are units modulo $q$ and this allows us to define a map
\[
\phi \colon \mathcal{F}_{\Q}^{\mathfrak{m}} \longrightarrow (\Z/q\Z)^{\times}, \quad \phi((a/b)) := a b^{-1} \pmod{q}.
\]
This map is a surjective group homomorphism. Moreover, its kernel is precisely $\mathcal{R}_{\mathbb{Q}}^{\mathfrak{m}}$. Indeed,
\[
ab^{-1}\equiv 1 \pmod q
\qquad\Longleftrightarrow\qquad
a\equiv b \pmod q.
\]
Since $p\nmid b$ for every $p\mid q$, this is equivalent to
\[
\operatorname{ord}_p\left(\frac{a}{b}-1\right)
=
\operatorname{ord}_p\left(\frac{a-b}{b}\right)
=
\operatorname{ord}_p(a-b)
\geq
\operatorname{ord}_p(q)
\qquad\text{for every }p\mid q.
\]
Together with the positivity of the generator $a/b$, this implies that
\[
\ker(\phi)=\mathcal{R}_{\mathbb{Q}}^{\mathfrak{m}}.
\]
Thus, the First Isomorphism Theorem for groups induces an isomorphism
\[
\overline{\phi} \colon \Cl_{\Q}^{\mathfrak{m}} \longrightarrow (\Z/q\Z)^{\times}, \quad \overline{\phi}((a/b) \mathcal{R}_{\Q}^{\mathfrak{m}}) := a b^{-1} \pmod{q}.
\]

Under this isomorphism, the ray class of the integral ideal $(n)$, with $(n,q)=1$, corresponds to the reduced residue class of $n$ modulo $q$. Thus, for $K=\mathbb{Q}$ and $\mathfrak{m}=(q)\cdot\infty$, ray classes are precisely reduced arithmetic progressions modulo $q$.
\end{example}

\subsection{Ray class zeta functions}

Let $\mathfrak{m}$ be a modulus for $K$, which will be fixed throughout this section.
Denote by $\mathcal{I}_K$ the set of nonzero integral ideals in $K$ and $\mathcal{I}_K^{\mathfrak{m}} := \{ \mathfrak{a} \in \mathcal{I}_K \suchthat \gcd{(\mathfrak{a}, \mathfrak{m}_0)} = 1 \}$.

For a ray class $\mathfrak{A} \in \Cl_K^{\mathfrak{m}}$ we define the associated \textit{ray class zeta function}
\begin{align}\label{eqn:ray-class-zeta-function}
\zeta_K(s;\mathfrak{A})
:=
\sum_{\substack{
\mathfrak{a}\in \mathcal{I}_K^{\mathfrak{m}}\\
[\mathfrak{a}]=\mathfrak{A}
}}
\frac{1}{\mathbf{N}_{K/\mathbb{Q}}(\mathfrak{a})^s}, \qquad \operatorname{Re}(s) > 1.
\end{align}
We also define the \textit{global zeta function}
\begin{align}
\label{eqn:global-zeta-function}
\zeta_K(s; \mathfrak{m}) := \sum_{\mathfrak{A}\in \mathrm{Cl}_K^{\mathfrak{m}}}
\zeta_K(s; \mathfrak{A}) = \sum_{\mathfrak{a} \in \mathcal{I}_K^{\mathfrak{m}}} \frac{1}{\Nrm{K/\Q}{\mathfrak{a}}^s}, \qquad \operatorname{Re}(s) > 1.
\end{align}
We observe that
\begin{align}\label{eqn:dedekind-zeta-and-zeta-modulo}
\zeta_K(s; \mathfrak{m}) = \prod_{\mathfrak{p} \nmid \mathfrak{m}_0} \left( 1 - \frac{1}{\Nrm{K/\Q}{\mathfrak{p}}^s} \right)^{-1} = \zeta_{K}(s) \prod_{\mathfrak{p} \mid \mathfrak{m}_0} \left( 1 - \frac{1}{\Nrm{K/\Q}{\mathfrak{p}}^s} \right),
\end{align}
so that $\zeta_K(s; \mathfrak{m})$ differs from the Dedekind Zeta function $\zeta_K(s)$ by a finite Euler product corresponding to the prime ideals that divide the finite part $\mathfrak{m}_0$.

\subsection{Ideal counting functions}\label{section:ideal-counting-functions}

We now define several sets of ideals and their corresponding counting functions. We will use uppercase German Fraktur letters like $\mathfrak{A}$ to denote ray classes and lowercase letters like $\mathfrak{a}$ to denote integral ideals. In certain cases we will also denote the ray class of the ideal $\mathfrak{a}$ by $[\mathfrak{a}]$.

Let $\mathfrak{A} \in \mathrm{Cl}_K^{\mathfrak{m}}$ be a ray class modulo $\mathfrak{m}$.
Recall that for any real number $X > 0$, there are only a finite number of integral ideals $\mathfrak{a} \trianglelefteq \mathcal{O}_K$ with norm $\Nrm{K/\Q}{\mathfrak{a}} \leq X$. 
Thus we define the (finite) sets
\begin{align*}
\mathcal{I}_K(X) &:= \{ \mathfrak{a} \in \mathcal{I}_K \suchthat \Nrm{K/\Q}{\mathfrak{a}} \leq X \},\\
\mathcal{I}_K^{\mathfrak{m}}(X) &:= \{ \mathfrak{a} \in \mathcal{I}_K^{\mathfrak{m}} \suchthat \Nrm{K/\Q}{\mathfrak{a}} \leq X \},\\
\mathcal{I}_K^{\mathfrak{m}}(X; \mathfrak{A}) &:= \left\{ \mathfrak{a} \in \mathcal{I}_K^{\mathfrak{m}} \suchthat \text{$\mathfrak{a} \in \mathfrak{A}$ and $\Nrm{K/\Q}{\mathfrak{a}} \leq X$} \right\}.
\end{align*}

\begin{example}
    \label{ex:classQ}
    Let $K = \Q$.
    Given a positive integer $q \geq 1$ consider the modulus $\mathfrak{m} = (q) \cdot \infty$.
    Let $a$ be a positive integer prime to $q$, and let $\mathfrak{A}$ denote the ray class of $(a)$.
    Then the map $n \mapsto (n)$ gives a bijection
   \begin{equation}
        \label{eqn:arith_progression}
       \{n \in \Z_{\geq 1} \suchthat n \equiv a \bmod{q}, \, n \leq X \}
       \longrightarrow
       \mathcal{I}^\mathfrak{m}_\Q(X;\mathfrak{A}),
   \end{equation} 
   which maps $k$-free integers to $k$-free ideals.
   Applying Theorem~\ref{thm:k-free-ideals-in-ray-classes} then gives \eqref{eqn:cohen}, since $\rho_\Q^\mathfrak{m} = 1/q$.
   Moreover, in this case Theorem~\ref{thm:main-theorem} agrees with \cite[Theorem 1]{nymann1992distribution}.
\end{example}

Now we consider the respective counting functions
\begin{align*}
\mathcal{H}_K(X) &:= \# \mathcal{I}_K(X),\\
\mathcal{H}_K^{\mathfrak{m}}(X) &:= \# \mathcal{I}_K^{\mathfrak{m}}(X),\\
\mathcal{H}_K^{\mathfrak{m}}(X; \mathfrak{A}) &:= \# \mathcal{I}_K^{\mathfrak{m}}(X; \mathfrak{A}).
\end{align*}

A classical result that combines work due to Dedekind (1894) and Weber (1897) gives the asymptotic growth of the counting function $\mathcal{H}_K(X)$ as
\begin{align}
\label{eqn:dedekind-weber}
\mathcal{H}_K(X) = \rho_K X + O(X^{1 - 1/d}),
\end{align}
where 
\begin{align*}
\rho_K = \underset{s=1}{\operatorname{Res}}\,\zeta_K(s)=\frac{2^{r_1}(2 \pi)^{r_2} \operatorname{Reg}_K h_K}{w_K\left|d_K\right|^{1 / 2}}.
\end{align*}
We now isolate the ideal-counting input that will be used in the proof of our
main theorem.

\begin{definition}\label{def:admissible-ideal-counting-exponent}
Let $\mathfrak{m}$ be a modulus for $K$. A real number
$0<\delta\leq1$ is called an \emph{admissible ideal-counting exponent}
for $(K,\mathfrak{m})$ if, for every ray class
$\mathfrak{A}\in\Cl_K^{\mathfrak{m}}$,
\begin{align}\label{eqn:admissible-ideal-counting}
\mathcal{H}_K^{\mathfrak{m}}(X;\mathfrak{A})
=
\rho_K^{\mathfrak{m}}X
+
O\left(X^{1-\delta}\right)
\end{align}
as $X\to\infty$. Since $\Cl_K^{\mathfrak{m}}$ is finite, the implied
constant may be chosen uniformly over all ray classes.
\end{definition}

The following theorem records the two admissible exponents that we use in
this paper. The general estimate is a standard ray-class counting result; a
proof can be found, for instance, in \cite[Chap. VI, \S 3, Thm. 3]{Lan94}.
The improvement for ordinary classes goes back to Landau
\cite[Satz 210]{landau1949}, while the corresponding result for narrow
classes can be found in \cite[p. 417]{narkiewicz1974}. In small degrees,
sharper estimates are known in special cases; for example, see
\cite{huxley1994number} for quadratic fields.

\begin{theorem}\label{thm:ideal-counting-admissible-exponents}
Let $K$ be a number field of degree $d=[K:\Q]$ and let $\mathfrak{m}$ be a
modulus for $K$.
\begin{enumerate}
\item For every modulus $\mathfrak{m}$, the exponent
\[
\delta=\frac{1}{d}
\]
is admissible for $(K,\mathfrak{m})$. More precisely, for every
$\mathfrak{A}\in\Cl_K^{\mathfrak{m}}$,
\begin{align}\label{eqn:ideal-counting-asymptotic-for-ray-classes}
\mathcal{H}_K^{\mathfrak{m}}(X;\mathfrak{A})
=
\rho_K^{\mathfrak{m}}X
+
O\left(X^{1-1/d}\right),
\end{align}
where
\begin{align}\label{eqn:residue}
\rho_K^{\mathfrak{m}}
=
\underset{s=1}{\operatorname{Res}}\,\zeta_K(s;\mathfrak{A})
=
\frac{2^{r_1}(2\pi)^{r_2}\mathrm{Reg}_{\mathfrak{m}}}
{w_{\mathfrak{m}}|d_K|^{1/2}2^{\#\mathfrak{m}_{\infty}}
\Nrm{K/\Q}{\mathfrak{m}_0}}.
\end{align}
Here $\mathrm{Reg}_{\mathfrak{m}}$ is the regulator of the subgroup
$\mathcal{O}_K^{\times}\cap K^{\mathfrak{m},1}$ and
$w_{\mathfrak{m}}$ is the number of roots of unity contained in this
subgroup. Summing over all ray classes gives
\begin{align}\label{eqn:ideal-counting-asymptotic-for-ideals-prime-to-m}
\mathcal{H}_K^{\mathfrak{m}}(X)
=
h_K^{\mathfrak{m}}\rho_K^{\mathfrak{m}}X
+
O\left(X^{1-1/d}\right).
\end{align}

\item Suppose that $d>1$ and that either
\begin{itemize}
    \item[(i)] $\mathfrak{m}(v)=0$ for every $v\in M_K$, so that
    $\Cl_K^{\mathfrak{m}}=\Cl_K$, or
    \item[(ii)] $\mathfrak{m}(v)=0$ for every $v\in M_K^0$ and
    $\mathfrak{m}(v)=1$ for every real $v\in M_K^\infty$, so that
    $\Cl_K^{\mathfrak{m}}=\Cl_K^+$.
\end{itemize}
Then
\[
\delta=\frac{2}{d+1}
\]
is also admissible for $(K,\mathfrak{m})$. Equivalently, for every
$\mathfrak{A}\in\Cl_K^{\mathfrak{m}}$,
\begin{align}\label{eqn:ideal-counting-asymptotic-for-ray-classes-improved}
\mathcal{H}_K^{\mathfrak{m}}(X;\mathfrak{A})
=
\rho_K^{\mathfrak{m}}X
+
O\left(X^{1-2/(d+1)}\right).
\end{align}
\end{enumerate}
\end{theorem}

\begin{remark}
The fact that
$\rho_K^{\mathfrak{m}}=\underset{s=1}{\operatorname{Res}}\,\zeta_K(s;\mathfrak{A})$
can be found in \cite[Chap. VIII, \S 2, Thm. 5c]{Lan94}. Also, an immediate
consequence of this and equation \eqref{eqn:global-zeta-function} is that
\[
\underset{s=1}{\operatorname{Res}}\,\zeta_K(s;\mathfrak{m})
=
h_K^{\mathfrak{m}}\rho_K^{\mathfrak{m}};
\]
see \cite[Chap. VIII, \S 2, Corollary to Thm. 5c, p. 161]{Lan94}.
\end{remark}

We now define the counting functions that will be our main interest in this article.

\begin{definition}\label{def:counting-functions}
Suppose that $k, r \in \Z_{\geq 1}$, and let $\mathfrak{A} \in \mathrm{Cl}_K^{\mathfrak{m}}$ be a ray class modulo $\mathfrak{m}$.
Then for any $X > 0$ we define the following functions.
\begin{enumerate}
    \item We let $Q_K^{k}(X)$ denote the number of ideals $\mathfrak{a} \in \mathcal{I}_K(X)$ that are $k$-free, that is, such that there does not exist a nonzero prime ideal $\mathfrak{p} \subseteq \mathcal{O}_K$ with $\mathfrak{p}^k \mid \mathfrak{a}$.
    
    \item We let $Q_K^{k}(X; \mathfrak{A})$ denote the number of ideals $\mathfrak{a} \in \mathcal{I}_K^{\mathfrak{m}}(X; \mathfrak{A})$ that are $k$-free.

    \item We let $Q_K^{k, r}(X)$ denote the number of $r$-tuples $(\mathfrak{a}_1, \dots, \mathfrak{a}_r) \in \mathcal{I}_K(X)^r$ of nonzero integral ideals $\mathfrak{a}_i \in \mathcal{I}_K(X)$ such that the ideals $\mathfrak{a}_1, \dots, \mathfrak{a}_r$ are relatively $k$-free, that is, there does not exist a nonzero prime ideal $\mathfrak{p} \subseteq \mathcal{O}_K$ such that $\mathfrak{p}^k \mid \mathfrak{a}_i$ for every $i = 1, \dots, r$.

    \item Given an $r$-tuple of ray classes $\boldsymbol{\mathfrak{A}} = (\mathfrak{A}_1,\dots,\mathfrak{A}_r) \in \left(\mathrm{Cl}_K^{\mathfrak{m}}\right)^r$, we let
    \[
    Q_K^{k,r}(X; \boldsymbol{\mathfrak{A}}) := Q_K^{k,r}(X;\mathfrak{A}_1,\dots,\mathfrak{A}_r)
    \]
    denote the number of relatively $k$-free $r$-tuples
    \[
    (\mathfrak{a}_1,\dots,\mathfrak{a}_r)
    \in
    \prod_{i=1}^r
    \mathcal{I}_K^{\mathfrak{m}}(X;\mathfrak{A}_i).
    \]

\end{enumerate}
\end{definition}

\begin{remark}\label{rem:sittinger-comparison}
The counting function $Q_K^{k,r}(X)$ was previously considered by
Sittinger \cite[Thm.~4.1]{Sit10} and, in a more general $\ell$-wise
setting, by Sittinger and DeMoss \cite[Thm.~2.1]{Sit20}. Their notation
differs from ours: the parameters corresponding to our $(r,k)$ are
denoted $(k,r)$ in \cite{Sit10}, while the case relevant to us in
\cite{Sit20} is obtained by taking their $\ell$-wise parameter equal to
the number $r$ of ideals in the tuple.

The main term in both papers agrees with the one obtained from Theorem~\ref{thm:main-theorem} by taking the trivial
modulus $\mathfrak{m}=(1)$ and summing over all $r$-tuples of ideal
classes, although the main term in \cite[Thm. 2.1]{Sit20} is expressed as an Euler product; in particular, the corresponding density is
\[
    \frac{1}{\zeta_K(rk)}.
\]
The difference in the error terms is caused by the ideal-counting estimates used: \cite{Sit10} uses the Dedekind--Weber estimate, corresponding to the admissible exponent $\delta=1/d$, whereas \cite{Sit20} uses Landau's sharper exponent $\delta=2/(d+1)$.

With these respective choices of $\delta$,
Theorem~\ref{thm:main-theorem} recovers the error terms obtained in
\cite{Sit10,Sit20} in their common range and improves them in certain cases. In particular, for $(r,k)=(2,1)$ and $d>1$, the error
term
\[
O\!\left(X^{r-2/(d+1)}\log X\right)
\]
appearing in \cite[Thm.~2.1]{Sit20} is improved to
\[
O\!\left(X^{r-2/(d+1)}\right).
\]
This improvement comes from Lemma~\ref{lem:several-cases-estimation}, where partial summation is used in place of a pointwise estimate for the number of ideals of a given norm.

In the case $r=1$, which is not covered by
\cite[Thm.~2.1]{Sit20}, we also obtain improvements over
\cite{Sit10}. Here the comparison involves different ideal-counting
exponents: for the trivial modulus, Theorem~\ref{thm:main-theorem}
may be applied with Landau's exponent $\delta=2/(d+1)$, whereas
\cite{Sit10} uses the Dedekind--Weber exponent $\delta=1/d$.
The corresponding error terms are summarized in Table \ref{tab:sittinger-r1-comparison}.

\begin{table}[h]
\centering
\renewcommand{\arraystretch}{1.35}
\begin{tabular}{c c c c}
\toprule
Degree & Range of $k$
& Our error term
& Error term in \cite{Sit10} \\
\midrule
$d=2$ & $k=2$
& $O\left(X^{1/2}\right)$
& $O\left(X^{3/4}\right)$ \\

$d=2$ & $k=3$
& $O\left(X^{1/3}\log X\right)$
& $O\left(X^{1/2}\log X\right)$ \\

$d=2$ & $k\geq4$
& $O\left(X^{1/3}\right)$
& $O\left(X^{1/2}\right)$ \\

$d=3$ & $k=2$
& $O\left(X^{1/2}\log X\right)$
& $O\left(X^{5/6}\right)$ \\

$d=3$ & $k\geq3$
& $O\left(X^{1/2}\right)$
& $O\left(X^{2/3}\right)$ \\

$d\geq4$ & $k=2$
& $O\left(X^{1-\frac{2}{d+1}}\right)$
& $O\left(X^{1-\frac{1}{2d}}\right)$ \\

$d\geq4$ & $k\geq3$
& $O\left(X^{1-\frac{2}{d+1}}\right)$
& $O\left(X^{1-\frac1d}\right)$ \\
\bottomrule\\
\end{tabular}
\caption{Comparison of the error terms for $r=1$ in
Theorem~\ref{thm:main-theorem} and in the proof of \cite[Theorem 4.1]{Sit10}, using Landau's exponent $\delta=2/(d+1)$ and the Dedekind--Weber exponent $\delta=1/d$, respectively.}
\label{tab:sittinger-r1-comparison}
\end{table}

Thus, for $r=1$ and $d\geq2$, the use of Landau's sharper
ideal-counting estimate yields a power-saving improvement over \cite{Sit10} in every case. 

\end{remark}

\section{\texorpdfstring{Equidistribution of relatively $k$-free tuples of ideals}{Equidistribution of relatively k-free tuples of ideals}} 

In this section, we prove the main results stated in the introduction.
We begin by expressing the counting function for relatively $k$-free tuples in terms of the M\"obius function and establishing the auxiliary
estimates needed to control the resulting sums. We then prove
Theorem~\ref{thm:main-theorem}, which gives the asymptotic formula for relatively $k$-free tuples lying in prescribed ray classes. Finally,
we specialize this result to the case $r=1$ to obtain
Theorems~\ref{thm:k-free-ideals-in-ray-classes} and
\ref{thm:k-free-ideals-class-and-narrow-class-groups}.

Throughout this section we fix a number field $K$ of degree $d$ and a modulus $\mathfrak{m}$ for $K$.
We also let $\mathfrak{A} \in \mathrm{Cl}_K^{\mathfrak{m}}$ be a fixed ray class modulo $\mathfrak{m}$ and if $\mathfrak{f} \in \mathcal{F}_K^{\mathfrak{m}}$ is a fractional ideal coprime to $\mathfrak{m}_0$, we denote its ray class in $\Cl_K^{\mathfrak{m}}$ by $[\mathfrak{f}]$.
For an integral ideal $\mathfrak{c} \in \mathcal{I}_K^{\mathfrak{m}}$ we define
\begin{align}\label{eqn:ideals-divisible-by-c}
\mathcal{I}_K^{\mathfrak{m}}(X; \mathfrak{A}, \mathfrak{c}) := \left\{ \mathfrak{a} \in \mathcal{I}_K^{\mathfrak{m}} \suchthat \text{$\mathfrak{a} \in \mathfrak{A}$,  $\mathfrak{c} \mid \mathfrak{a}$  and $\Nrm{K/\Q}{\mathfrak{a}} \leq X$} \right\}
\end{align}
and its corresponding counting function
\begin{align*}
\mathcal{H}_K^{\mathfrak{m}}(X; \mathfrak{A}, \mathfrak{c}) &:= \# \mathcal{I}_K^{\mathfrak{m}}(X; \mathfrak{A}, \mathfrak{c}).
\end{align*}

\begin{lemma}\label{lem:number-of-ideals-in-a-ray-class-divisible-by-an-ideal}
Let $\mathfrak{c} \in \mathcal{I}_K^{\mathfrak{m}}$.
Then for any $X > 0$ there is a bijection
\begin{align*}
f \colon \mathcal{I}_K^{\mathfrak{m}} \left(\frac{X}{\Nrm{K/\Q}{\mathfrak{c}}}; \mathfrak{A} \cdot [\mathfrak{c}^{-1}] \right) \longrightarrow \mathcal{I}_K^{\mathfrak{m}}(X; \mathfrak{A}, \mathfrak{c}),
\end{align*}
defined by $f(\mathfrak{b}) := \mathfrak{c} \cdot \mathfrak{b}$.
In particular, 
\begin{align*}
\mathcal{H}_K^{\mathfrak{m}}(X; \mathfrak{A}, \mathfrak{c}) 
= \mathcal{H}_K^{\mathfrak{m}}\left( \dfrac{X}{\Nrm{K/\Q}{\mathfrak{c}}}; \mathfrak{A} \cdot[\mathfrak{c}^{-1}] \right).
\end{align*}
\end{lemma}

\begin{proof}
First, observe that the function $f$ is well-defined, and note that its domain is empty if and only if the codomain is empty. This occurs for sufficiently small $X$, in which case the lemma holds trivially. We may therefore assume that $X$ is sufficiently large so that both the domain and codomain of $f$ are nonempty.

To see that $f$ is injective note that if $f(\mathfrak{b}_1) = f(\mathfrak{b}_2)$ for some $\mathfrak{b}_1, \mathfrak{b}_2 \in \mathcal{I}_K^{\mathfrak{m}} \left(\frac{X}{\Nrm{K/\Q}{\mathfrak{c}}}; \mathfrak{A} \cdot [\mathfrak{c}^{-1}] \right)$, then $\mathfrak{c} \cdot \mathfrak{b}_1 = \mathfrak{c} \cdot \mathfrak{b}_2$ and this implies that $\mathfrak{b}_1 = \mathfrak{b}_2$. To see that $f$ is surjective, let $\mathfrak{a} \in \mathcal{I}_K^{\mathfrak{m}}(X; \mathfrak{A}, \mathfrak{c})$. Then $\mathfrak{c}^{-1} \cdot \mathfrak{a} \in \mathcal{I}_K^{\mathfrak{m}} \left(\frac{X}{\Nrm{K/\Q}{\mathfrak{c}}}; \mathfrak{A} \cdot [\mathfrak{c}^{-1}] \right)$ and we have $f(\mathfrak{c}^{-1} \cdot \mathfrak{a}) = \mathfrak{c} \cdot (\mathfrak{c}^{-1} \cdot \mathfrak{a}) = \mathfrak{a}$.
\end{proof}

We will need the following version of the Inclusion-Exclusion Principle.

\begin{proposition}[Inclusion-Exclusion Principle]\label{prop:inclusion-exclusion}
Let \( U \) be a finite set, and let \( \{ A_i \}_{i \in I} \) be a finite collection of subsets of \( U \).
Then the size of the complement of the union of the sets \( A_i \) is given by
\[
\left| U \smallsetminus \bigcup_{i \in I} A_i \right| 
= \sum_{J \subseteq I} (-1)^{|J|} \left| \bigcap_{j \in J} A_j \right|.
\]
\end{proposition}

\begin{remark}
We observe that in the previous proposition, the intersection over the empty set is defined to be the universal set, that is
\[
\bigcap_{j \in \varnothing} A_j := U.
\]
\end{remark}

Now, we recall that the Möbius function $\mu_K \colon \mathcal{I}_K \longrightarrow \{ 0, \pm 1 \}$ for the number field $K$ is defined for any nonzero integral $\mathfrak{a} \in \mathcal{I}_K$ by
\begin{align*}
\mu_K(\mathfrak{a}) = 
\begin{cases}1 & \text { if } \mathfrak{a} = \mathcal{O}_K, \\ 0 & \text { if } \mathfrak{p}^2 \mid \mathfrak{a} \text { for some prime ideal } \mathfrak{p}, \\ (-1)^t & \text{ if } \mathfrak{a} = \mathfrak{p}_1 \mathfrak{p}_2 \cdots \mathfrak{p}_t \text { for distinct primes } \mathfrak{p}_1, \mathfrak{p}_2, \ldots, \mathfrak{p}_t.
\end{cases}
\end{align*}

In the next proposition we use the Inclusion-Exclusion Principle stated above, the Möbius function $\mu_K$ and the formula from Lemma \ref{lem:number-of-ideals-in-a-ray-class-divisible-by-an-ideal} to obtain a formula for the counting function $Q_K^{k, r}(X; \mathfrak{A})$.

\begin{proposition}\label{prop:Q-identity-with-moebius}
Let $k,r\in\Z_{\geq 1}$ and let $\boldsymbol{\mathfrak{A}} = (\mathfrak{A}_1,\dots,\mathfrak{A}_r) \in \left(\mathrm{Cl}_K^{\mathfrak{m}}\right)^r$.
Then, for every $X>0$, we have
\begin{align*}
Q_K^{k,r}(X;\boldsymbol{\mathfrak{A}})
&=
\sum_{\mathfrak{a}\in
\mathcal{I}_K^{\mathfrak{m}}(\sqrt[k]{X})}
\mu_K(\mathfrak{a})
\prod_{i=1}^{r}
\mathcal{H}_K^{\mathfrak{m}}
\left(
\frac{X}{\Nrm{K/\Q}{\mathfrak{a}^k}};
\mathfrak{A}_i\cdot[\mathfrak{a}^{-k}]
\right).
\end{align*}
\end{proposition}

\begin{proof}
We apply the Inclusion--Exclusion Principle to count the relatively
$k$-free $r$-tuples whose $i$-th entry lies in the ray class
$\mathfrak{A}_i$. Let
\[
U
:=
\prod_{i=1}^{r}
\mathcal{I}_K^{\mathfrak{m}}(X;\mathfrak{A}_i)
\]
be the set of all $r$-tuples
$(\mathfrak{a}_1,\dots,\mathfrak{a}_r)$ of integral ideals such that
$\mathfrak{a}_i\in\mathfrak{A}_i$ and
$\Nrm{K/\Q}{\mathfrak{a}_i}\leq X$ for every
$1\leq i\leq r$. We also let
\[
I
:=
\mathbb{P}_K(\sqrt[k]{X};\mathfrak{m})
\]
be the set of all prime ideals $\mathfrak{p}$ of $K$ such that
$\Nrm{K/\Q}{\mathfrak{p}}\leq \sqrt[k]{X}$ and
$\mathfrak{p}\nmid\mathfrak{m}_0$. Then, using the notation from equation \eqref{eqn:ideals-divisible-by-c}, for every $\mathfrak{p}\in I$,
define
\[
A_{\mathfrak{p}}
:=
\prod_{i=1}^{r}
\mathcal{I}_K^{\mathfrak{m}}
(X;\mathfrak{A}_i,\mathfrak{p}^k).
\]
Thus, $A_{\mathfrak{p}}$ is the set of all tuples
$(\mathfrak{a}_1,\dots,\mathfrak{a}_r)\in U$ such that
$\mathfrak{p}^k\mid\mathfrak{a}_i$ for every $1\leq i\leq r$.

It follows that
\[
U\smallsetminus\bigcup_{\mathfrak{p}\in I}A_{\mathfrak{p}}
\]
is precisely the set of relatively $k$-free tuples
\[
(\mathfrak{a}_1,\dots,\mathfrak{a}_r)
\in
\prod_{i=1}^{r}
\mathcal{I}_K^{\mathfrak{m}}(X;\mathfrak{A}_i).
\]
Therefore,
\begin{align}\label{eqn:Q-tuples-complement-formula}
Q_K^{k,r}(X;\boldsymbol{\mathfrak{A}})
=
\left|
U\smallsetminus
\bigcup_{\mathfrak{p}\in I}A_{\mathfrak{p}}
\right|.
\end{align}

For every subset $J\subseteq I$, define
\[
\mathfrak{c}_J
:=
\prod_{\mathfrak{p}\in J}\mathfrak{p},
\]
where $\mathfrak{c}_{\varnothing}:=\mathcal{O}_K$. Then
\begin{align*}
\bigcap_{\mathfrak{p}\in J}A_{\mathfrak{p}}
&=
\prod_{i=1}^{r}
\mathcal{I}_K^{\mathfrak{m}}
(X;\mathfrak{A}_i,\mathfrak{c}_J^k),
\end{align*}
since a tuple belongs to this intersection if and only if each of its
entries is divisible by $\mathfrak{p}^k$ for every $\mathfrak{p}\in J$,
or equivalently, by $\mathfrak{c}_J^k$. Hence,
Lemma~\ref{lem:number-of-ideals-in-a-ray-class-divisible-by-an-ideal}
implies that
\begin{align}\label{eqn:intersection-Ap-size}
\left|
\bigcap_{\mathfrak{p}\in J}A_{\mathfrak{p}}
\right|
&=
\prod_{i=1}^{r}
\mathcal{H}_K^{\mathfrak{m}}
\left(
\frac{X}{\Nrm{K/\Q}{\mathfrak{c}_J^k}};
\mathfrak{A}_i\cdot[\mathfrak{c}_J^{-k}]
\right).
\end{align}

Now, observe that for every $J\subseteq I$ we have $\mu_K(\mathfrak{c}_J)=(-1)^{|J|}$. Thus, by the Inclusion--Exclusion Principle from
Proposition~\ref{prop:inclusion-exclusion} and
equations~\eqref{eqn:Q-tuples-complement-formula}
and~\eqref{eqn:intersection-Ap-size}, we obtain
\begin{align*}
Q_K^{k,r}(X;\boldsymbol{\mathfrak{A}})
&=
\left|
U\smallsetminus
\bigcup_{\mathfrak{p}\in I}A_{\mathfrak{p}}
\right|
=
\sum_{J\subseteq I}
(-1)^{|J|}
\left|
\bigcap_{\mathfrak{p}\in J}A_{\mathfrak{p}}
\right|\\
&=
\sum_{J\subseteq I}
(-1)^{|J|}
\prod_{i=1}^{r}
\mathcal{H}_K^{\mathfrak{m}}
\left(
\frac{X}{\Nrm{K/\Q}{\mathfrak{c}_J^k}};
\mathfrak{A}_i\cdot[\mathfrak{c}_J^{-k}]
\right).
\end{align*}

The map $J\mapsto\mathfrak{c}_J$ identifies the subsets of $I$ with
the squarefree ideals all of whose prime ideal divisors belong to $I$.
Moreover, if
$\Nrm{K/\Q}{\mathfrak{c}_J}>\sqrt[k]{X}$, then
\[
\frac{X}{\Nrm{K/\Q}{\mathfrak{c}_J^k}}<1,
\]
so the corresponding term vanishes. Finally,
$\mu_K(\mathfrak{a})=0$ for every nonsquarefree ideal $\mathfrak{a}$.
Consequently, the preceding sum may be written as
\begin{align*}
Q_K^{k,r}(X;\boldsymbol{\mathfrak{A}})
&=
\sum_{\mathfrak{a}\in
\mathcal{I}_K^{\mathfrak{m}}(\sqrt[k]{X})}
\mu_K(\mathfrak{a})
\prod_{i=1}^{r}
\mathcal{H}_K^{\mathfrak{m}}
\left(
\frac{X}{\Nrm{K/\Q}{\mathfrak{a}^k}};
\mathfrak{A}_i\cdot[\mathfrak{a}^{-k}]
\right),
\end{align*}
which proves the result.
\end{proof}

\begin{lemma}\label{lem:several-cases-estimation}
Let $s\geq 0$ and $T > 1$. Then
\begin{align*}
\sum_{\mathfrak{a}\in\mathcal{I}_K^{\mathfrak{m}}(T)}
\frac{1}{\Nrm{K/\Q}{\mathfrak{a}}^s}
=
\begin{cases}
O(T^{1-s}), & \text{if $0\leq s<1$},\\[4pt]
O(\log T), & \text{if $s=1$},\\[4pt]
O(1), & \text{if $s>1$}.
\end{cases}
\end{align*}
\end{lemma}

\begin{proof}
For each $n\geq1$, let
\[
a_n
:=
\#\left\{
\mathfrak{a}\in\mathcal{I}_K^{\mathfrak{m}}
\suchthat
\Nrm{K/\Q}{\mathfrak{a}}=n
\right\}.
\]
Then its summatory function is
\[
A(x)
:=
\sum_{n\leq x}a_n
=
\mathcal{H}_K^{\mathfrak{m}}(x).
\]
By \eqref{eqn:ideal-counting-asymptotic-for-ideals-prime-to-m}, we have
\[
\mathcal{H}_K^{\mathfrak{m}}(x)
=
h_K^{\mathfrak{m}} \rho_K^{\mathfrak{m}}x + O\left(x^{1-\delta}\right)
\]
for some $\delta>0$, and in particular
\begin{equation}\label{eqn:H-linear-bound}
A(x)=\mathcal{H}_K^{\mathfrak{m}}(x)=O(x).
\end{equation}

We now apply partial summation (see for example \cite[\S 1.5]{iwaniec-kowalski2004} or \cite[Chapter 1]{montgomery-vaughan2007}). Recall that if
\[
A(x)=\sum_{n\leq x}a_n
\]
and $f$ is continuously differentiable, then
\begin{equation}\label{eqn:partial-summation-formula}
\sum_{n\leq T}a_n f(n)
=
A(T)f(T)
-
\int_1^T A(t)f'(t)\,dt.
\end{equation}

For $s>0$, we take $f(t)=t^{-s}$, so that $f'(t)=-s\,t^{-s-1}$.
Since
\[
\sum_{\mathfrak{a}\in\mathcal{I}_K^{\mathfrak{m}}(T)}
\frac{1}{\Nrm{K/\Q}{\mathfrak{a}}^s}
=
\sum_{n\leq T}\frac{a_n}{n^s},
\]
formula \eqref{eqn:partial-summation-formula} gives
\begin{align}
\sum_{\mathfrak{a}\in\mathcal{I}_K^{\mathfrak{m}}(T)}
\frac{1}{\Nrm{K/\Q}{\mathfrak{a}}^s}
&=
\frac{\mathcal{H}_K^{\mathfrak{m}}(T)}{T^s}
+
s\int_1^T
\frac{\mathcal{H}_K^{\mathfrak{m}}(t)}
{t^{s+1}}\,dt.
\label{eqn:partial-summation-ideal-sum}
\end{align}
Using \eqref{eqn:H-linear-bound}, we obtain
\begin{align*}
\sum_{\mathfrak{a}\in\mathcal{I}_K^{\mathfrak{m}}(T)}
\frac{1}{\Nrm{K/\Q}{\mathfrak{a}}^s}
&\ll
T^{1-s}
+
\int_1^T t^{-s}\,dt.
\end{align*}

We now consider the three cases. If $0<s<1$, then
\[
\int_1^T t^{-s}\,dt
=
\frac{T^{1-s}-1}{1-s}
=
O(T^{1-s}),
\]
and therefore
\[
\sum_{\mathfrak{a}\in\mathcal{I}_K^{\mathfrak{m}}(T)}
\frac{1}{\Nrm{K/\Q}{\mathfrak{a}}^s}
=
O(T^{1-s}).
\]
If $s=1$, then
\[
\int_1^T\frac{dt}{t}=\log T,
\]
and hence the sum is $O(\log T)$. Finally, if $s>1$, then
\[
\int_1^T t^{-s}\,dt
\leq
\int_1^\infty t^{-s}\,dt
=
\frac{1}{s-1},
\]
while $T^{1-s}=O(1)$, and therefore the sum is $O(1)$. Finally, for $s=0$, the result follows directly from
$$
\sum_{\mathfrak{a}\in\mathcal{I}_K^{\mathfrak{m}}(T)}1
=
\mathcal{H}_K^{\mathfrak{m}}(T)
=
O(T).
$$
\end{proof}

\begin{lemma}\label{lem:finite-sum-with-moebius-to-dedekind-zeta}
Let $s>1$ and $T\geq 1$. Then
\begin{align*}
\sum_{\mathfrak{a}\in\mathcal{I}_K^{\mathfrak{m}}(T)}
\frac{\mu_K(\mathfrak{a})}
{\Nrm{K/\Q}{\mathfrak{a}}^s}
&=
\frac{1}{\zeta_K(s)}
\prod_{\mathfrak{p}\mid\mathfrak{m}_0}
\left(
1-\frac{1}{\Nrm{K/\Q}{\mathfrak{p}}^s}
\right)^{-1}
+
O(T^{1-s}).
\end{align*}
\end{lemma}

\begin{proof}
Since $s>1$, the series is absolutely convergent and
\begin{align*}
\sum_{\mathfrak{a}\in\mathcal{I}_K^{\mathfrak{m}}}
\frac{\mu_K(\mathfrak{a})}
{\Nrm{K/\Q}{\mathfrak{a}}^s}
&=
\frac{1}{\zeta_K(s;\mathfrak{m})}\\
&=
\frac{1}{\zeta_K(s)}
\prod_{\mathfrak{p}\mid\mathfrak{m}_0}
\left(
1-\frac{1}{\Nrm{K/\Q}{\mathfrak{p}}^s}
\right)^{-1}.
\end{align*}
Thus, we have
\begin{align*}
\left|
\sum_{\mathfrak{a}\in\mathcal{I}_K^{\mathfrak{m}}(T)}
\frac{\mu_K(\mathfrak{a})}
{\Nrm{K/\Q}{\mathfrak{a}}^s}
- \frac{1}{\zeta_K(s;\mathfrak{m})} \right|
\leq
\sum_{\substack{\mathfrak{a}\in\mathcal{I}_K^{\mathfrak{m}}\\
\Nrm{K/\Q}{\mathfrak{a}}>T}}
\frac{1}{\Nrm{K/\Q}{\mathfrak{a}}^s}.
\end{align*}

Now, to estimate this tail, we apply again partial summation as in the proof of Lemma \ref{lem:several-cases-estimation} and the bound $\mathcal{H}_K^{\mathfrak{m}}(X) =  O(X)$, and obtain
\begin{align*}
\sum_{\substack{\mathfrak{a}\in\mathcal{I}_K^{\mathfrak{m}}\\
\Nrm{K/\Q}{\mathfrak{a}}>T}}
\frac{1}{\Nrm{K/\Q}{\mathfrak{a}}^s}
&\ll
\frac{\mathcal{H}_K^{\mathfrak{m}}(T)}{T^s}
+
\int_T^\infty
\frac{\mathcal{H}_K^{\mathfrak{m}}(t)}{t^{s+1}}\,dt\\
&\ll
T^{1-s}
+
\int_T^\infty t^{-s}\,dt\\
&=
O(T^{1-s}),
\end{align*}
since $s>1$. This proves the result.
\end{proof}

\subsection{Proof of the main theorem}

\begin{proof}[Proof of Theorem~\ref{thm:main-theorem}]
Let $\delta$ be the admissible ideal-counting exponent appearing in the
statement of the theorem. By Definition~\ref{def:admissible-ideal-counting-exponent},
uniformly for $\mathfrak{C}\in\Cl_K^{\mathfrak{m}}$,
\begin{equation}\label{eqn:unified-ideal-counting}
\mathcal{H}_K^{\mathfrak{m}}(Y;\mathfrak{C})
=
\rho_K^{\mathfrak{m}}Y
+
O\left(Y^{1-\delta}\right).
\end{equation}
For simplicity, write
\[
\mathcal{J}
=
\mathcal{I}_K^{\mathfrak{m}}\left(X^{1/k}\right).
\]
By Proposition~\ref{prop:Q-identity-with-moebius}, we have
\begin{align}\label{eqn:Q-equality-with-moebius-for-proof}
Q_K^{k,r}(X;\boldsymbol{\mathfrak{A}})
&=
\sum_{\mathfrak{a}\in\mathcal{J}}
\mu_K(\mathfrak{a})
\prod_{i=1}^{r}
\mathcal{H}_K^{\mathfrak{m}}
\left(
\frac{X}{\Nrm{K/\Q}{\mathfrak{a}^k}};
\mathfrak{A}_i\cdot[\mathfrak{a}^{-k}]
\right).
\end{align}

Since the ray class group $\mathrm{Cl}_K^{\mathfrak{m}}$ is finite,
the implied constant in \eqref{eqn:unified-ideal-counting} may be
chosen uniformly over all ray classes. Hence, for any
$\mathfrak{C}_1,\ldots,\mathfrak{C}_r
\in\mathrm{Cl}_K^{\mathfrak{m}}$ and any $Y\geq1$,
\begin{align}
\prod_{i=1}^{r}
\mathcal{H}_K^{\mathfrak{m}}(Y;\mathfrak{C}_i)
&=
\prod_{i=1}^{r}
\left(
\rho_K^{\mathfrak{m}}Y
+
O\left(Y^{1-\delta}\right)
\right)
\notag\\
&=
\left(\rho_K^{\mathfrak{m}}\right)^rY^r
+
O\left(Y^{r-\delta}\right).
\label{eqn:product-H-asymptotic-unified}
\end{align}
Indeed, every term other than the main product has the form
\[
\left(\rho_K^{\mathfrak{m}}Y\right)^t
O\left(Y^{1-\delta}\right)^{r-t}
\]
for some $0\leq t\leq r-1$, and hence has order
\[
O\left(
Y^{t+(r-t)(1-\delta)}
\right)
=
O\left(
Y^{r-(r-t)\delta}
\right).
\]
The largest exponent occurs for $t=r-1$, giving
$O(Y^{r-\delta})$.

Next, applying \eqref{eqn:product-H-asymptotic-unified} with the choices
\[
Y
=
\frac{X}{\Nrm{K/\Q}{\mathfrak{a}^k}} \quad \text{and} \quad \mathfrak{C}_i
=
\mathfrak{A}_i\cdot[\mathfrak{a}^{-k}]
\]
we obtain
\begin{align}\label{eqn:product-H-asymptotic-substitution}
\prod_{i=1}^{r}
\mathcal{H}_K^{\mathfrak{m}}
\left(
\frac{X}{\Nrm{K/\Q}{\mathfrak{a}^k}};
\mathfrak{A}_i\cdot[\mathfrak{a}^{-k}]
\right) =
\left(\rho_K^{\mathfrak{m}}\right)^r
\frac{X^r}
{\Nrm{K/\Q}{\mathfrak{a}}^{rk}}
+
O\left(
\frac{X^{r-\delta}}
{\Nrm{K/\Q}{\mathfrak{a}}^{k(r-\delta)}}
\right).
\end{align}
Then, using \eqref{eqn:product-H-asymptotic-substitution} in
\eqref{eqn:Q-equality-with-moebius-for-proof} yields
\begin{align}
Q_K^{k,r}(X;\boldsymbol{\mathfrak{A}})
&=
\left(\rho_K^{\mathfrak{m}}\right)^r
\left( \sum_{\mathfrak{a}\in\mathcal{J}}
\frac{\mu_K(\mathfrak{a})}
{\Nrm{K/\Q}{\mathfrak{a}}^{rk}}
\right)X^r + O\left( \sum_{\mathfrak{a}\in\mathcal{J}}
\frac{1}{\Nrm{K/\Q}{\mathfrak{a}}^{k(r-\delta)}}\right)
X^{r-\delta}.
\label{eqn:Q-expression-with-two-sums}
\end{align}

We first consider the sum containing the M\"obius function.
Since $(k,r)\neq(1,1)$, we have $rk>1$. Therefore, using
Lemma \ref{lem:finite-sum-with-moebius-to-dedekind-zeta} with the choices $s=rk$ and $T=X^{1/k}$ gives
\begin{align}
\sum_{\mathfrak{a}\in\mathcal{J}}
\frac{\mu_K(\mathfrak{a})}
{\Nrm{K/\Q}{\mathfrak{a}}^{rk}}
&=
\frac{1}{\zeta_K(rk)}
\prod_{\mathfrak{p}\mid\mathfrak{m}_0}
\left(
1-\frac{1}
{\Nrm{K/\Q}{\mathfrak{p}}^{rk}}
\right)^{-1}
+
O\left(X^{1/k - r}\right).
\end{align}
Consequently,
\begin{align}
\left(\rho_K^{\mathfrak{m}}\right)^r
\left(
\sum_{\mathfrak{a}\in\mathcal{J}}
\frac{\mu_K(\mathfrak{a})}
{\Nrm{K/\Q}{\mathfrak{a}}^{rk}}
\right)X^r =
\frac{\left(\rho_K^{\mathfrak{m}}\right)^r}
{\zeta_K(rk)}
\prod_{\mathfrak{p}\mid\mathfrak{m}_0}
\left(
1-\frac{1}
{\Nrm{K/\Q}{\mathfrak{p}}^{rk}}
\right)^{-1} X^r + O\left(X^{1/k}\right).
\label{eqn:main-sum-asymptotic}
\end{align}

We next estimate the second sum in \eqref{eqn:Q-expression-with-two-sums}. Set $s:=k(r-\delta)$. Since $r\geq1$ and $0<\delta\leq1$, we have $s\geq0$, so
Lemma \ref{lem:several-cases-estimation} applies with
$T=X^{1/k}$. Hence
\begin{align}
&X^{r-\delta}
\sum_{\mathfrak{a}\in\mathcal{J}}
\frac{1}
{\Nrm{K/\Q}{\mathfrak{a}}^{k(r-\delta)}}
=
\begin{cases}
\displaystyle
O\left(
X^{r-\delta}
X^{\frac{1-k(r-\delta)}{k}}
\right)
=
O\left(X^{1/k}\right),
&
k(r-\delta)<1,
\\[10pt]
\displaystyle
O\left(
X^{r-\delta}\log X
\right),
&
k(r-\delta)=1,
\\[10pt]
\displaystyle
O\left(
X^{r-\delta}
\right),
&
k(r-\delta)>1.
\end{cases}
\label{eqn:secondary-error-cases}
\end{align}

It remains only to compare these estimates with the error
$O(X^{1/k})$ arising from the main sum in
\eqref{eqn:main-sum-asymptotic}. We claim that, in each of the three
cases above, the error term coming from the second sum in
\eqref{eqn:Q-expression-with-two-sums} is at least as large as
$O(X^{1/k})$, and therefore determines the final error term. Indeed,
if $k(r-\delta)<1$, the two error terms are both $O(X^{1/k})$; if
$k(r-\delta)=1$, then $r-\delta=1/k$, so
$O(X^{1/k})$ is absorbed by
$O(X^{r-\delta}\log X)$; and if $k(r-\delta)>1$, then
$r-\delta>1/k$, so $O(X^{1/k})$ is absorbed by
$O(X^{r-\delta})$.

Combining these estimates proves \eqref{eqn:main-asymptotic-formula} and \eqref{eqn:main-error-term}.
\end{proof}

\begin{remark}\label{rem:error-r-at-least-two}
Let $\delta$ be the admissible exponent used in
Theorem~\ref{thm:main-theorem}. Whenever $k(r-\delta)>1$, the third case of
\eqref{eqn:main-error-term} gives
\[
E_{K,\mathfrak{m},k,r}(X)=O\left(X^{r-\delta}\right).
\]
In particular, suppose that $d\geq2$ and $r\geq2$, and take either of the two
admissible exponents used in this paper, namely $\delta=1/d$ or, when
available, $\delta=2/(d+1)$. Then $\delta<1$ and
\[
k(r-\delta)\geq r-\delta>1,
\]
so the error is $O(X^{r-\delta})$. Thus, for these standard choices, all of
the exceptional cases in \eqref{eqn:main-error-term} occur when $r=1$.
\end{remark}

\subsection{Proofs of the specialized results}

We conclude by deriving Theorems \ref{thm:k-free-ideals-in-ray-classes} and \ref{thm:k-free-ideals-class-and-narrow-class-groups} and Corollary~\ref{cor:gegenbauer-cohen-robinson} from Theorem \ref{thm:main-theorem}. In both cases, we take $r=1$ and choose the appropriate admissible ideal-counting exponent $\delta$.

\subsubsection{Proof of Theorem~\ref{thm:k-free-ideals-in-ray-classes}}

For the general ray-class estimate, Theorem~\ref{thm:ideal-counting-admissible-exponents} shows that
$\delta=1/d$ is admissible for every modulus. We therefore take $r=1$ and
\[
\delta=\frac1d
\]
in Theorem~\ref{thm:main-theorem}. Thus, the relevant comparison for the error term is
\[
k\left(1-\frac1d\right)<1,\qquad
k\left(1-\frac1d\right)=1,\qquad\text{or}\qquad
k\left(1-\frac1d\right)>1.
\]

If $d=1$, then
\[
k\left(1-\frac1d\right)=0<1,
\]
and hence
\[
E_{\Q,\mathfrak{m},k}(X)=O(X^{1/k}).
\]

Suppose now that $d=2$. Then
\[
k\left(1-\frac12\right)=\frac{k}{2}.
\]
Therefore, equality occurs when $k=2$, giving
\[
E_{K,\mathfrak{m},2}(X)
=
O(X^{1/2}\log X),
\]
while for $k\geq3$ we have $k/2>1$, and hence
\[
E_{K,\mathfrak{m},k}(X)
=
O(X^{1/2}).
\]

Finally, if $d\geq3$, then for every $k\geq2$,
\[
k\left(1-\frac1d\right)
\geq
2\left(1-\frac1d\right)
>1.
\]
It follows that
\[
E_{K,\mathfrak{m},k}(X)
=
O\left(X^{1-\frac1d}\right).
\]
This proves Theorem~\ref{thm:k-free-ideals-in-ray-classes}. \qed

\subsubsection{Proof of Corollary~\ref{cor:gegenbauer-cohen-robinson}}

For part~(ii), take $K=\Q$ and
\[
\mathfrak{m}=(q)\infty.
\]
As explained in Example~\ref{ex:ray-class-groups-over-Q}, the map
$n\mapsto(n)$ identifies the positive integers $n\equiv a\pmod q$
with the integral ideals lying in the ray class corresponding to
$a\pmod q$, and preserves the property of being $k$-free.

It remains to compute the constant in
Theorem~\ref{thm:k-free-ideals-in-ray-classes}. In this case,
$r_1=1$, $r_2=0$, $|d_{\Q}|=1$,
$\#\mathfrak{m}_{\infty}=1$, and
\[
\Nrm{\Q/\Q}{\mathfrak{m}_0}=q.
\]
Moreover,
\[
\mathcal{O}_{\Q}^{\times}\cap\Q^{\mathfrak{m},1}=\{1\},
\]
so that $\operatorname{Reg}_{\mathfrak{m}}=w_{\mathfrak{m}}=1$.
Hence, by \eqref{eqn:residue},
\[
\rho_{\Q}^{\mathfrak{m}}
=
\frac{2}{2q}
=
\frac{1}{q}.
\]
Also, the prime ideals of $\Z$ dividing $(q)$ are precisely the ideals
$(p)$ with $p\mid q$, and
$\Nrm{\Q/\Q}{(p)}=p$. Therefore
\[
\prod_{\mathfrak{p}\mid\mathfrak{m}_0}
\left(
1-\frac{1}{\Nrm{\Q/\Q}{\mathfrak{p}}^k}
\right)^{-1}
=
\prod_{p\mid q}
\left(1-\frac{1}{p^k}\right)^{-1}.
\]
Since $\zeta_{\Q}(s)=\zeta(s)$, the result follows from
Theorem~\ref{thm:k-free-ideals-in-ray-classes}, whose error term in
this case is $O(X^{1/k})$.

Part~(i) follows by taking $q=1$, in which case
$\rho_{\Q}^{\mathfrak{m}}=1$ and the finite Euler product is empty. \qed

\subsubsection{Proof of
Theorem~\ref{thm:k-free-ideals-class-and-narrow-class-groups}}

For the ordinary and narrow class groups,
Theorem~\ref{thm:ideal-counting-admissible-exponents} shows that
$\delta=2/(d+1)$ is admissible. We again take $r=1$ in
Theorem~\ref{thm:main-theorem}, with
\[
\delta=\frac{2}{d+1}.
\]
The quantity $k(r-\delta)$ appearing in Theorem~\ref{thm:main-theorem}
becomes
\[
k\left(1-\frac{2}{d+1}\right)
=
\frac{k(d-1)}{d+1}.
\]
Hence, the relevant comparison for the error term is
\[
\frac{k(d-1)}{d+1}<1,\qquad
\frac{k(d-1)}{d+1}=1,\qquad\text{or}\qquad
\frac{k(d-1)}{d+1}>1.
\]

If $d=2$, the three cases are determined by whether
\[
\frac{k}{3}<1,\qquad
\frac{k}{3}=1,\qquad\text{or}\qquad
\frac{k}{3}>1.
\]
Thus, for $k=2$ we are in the first case of
Theorem~\ref{thm:main-theorem}, giving
\[
E(X)=O(X^{1/2});
\]
for $k=3$ we are in the equality case, giving
\[
E(X)=O(X^{1/3}\log X);
\]
and for $k\geq4$ we obtain
\[
E(X)=O(X^{1/3}).
\]

If $d=3$, then
\[
\frac{k(d-1)}{d+1}=\frac{k}{2}.
\]
Hence $k=2$ gives the equality case and therefore
\[
E(X)=O(X^{1/2}\log X),
\]
whereas for every $k\geq3$ we have
\[
E(X)=O(X^{1/2}).
\]

Finally, suppose that $d\geq4$. For every $k\geq2$,
\[
\frac{k(d-1)}{d+1}
\geq
\frac{2(d-1)}{d+1}
>1.
\]
Therefore,
\[
E(X)
=
O\left(X^{1-\frac{2}{d+1}}\right).
\]

In the ordinary class-group case, the finite part of the modulus is
trivial and
\[
\rho_K^{(1)}=\frac{\rho_K}{h_K},
\]
while in the narrow class-group case
\[
\rho_K^{\mathfrak{m}_{\infty}}
=
\frac{\rho_K}{h_K^+}.
\]
Substituting these identities into the main term in
Theorem~\ref{thm:main-theorem} gives the two asymptotic formulas stated
in Theorem~\ref{thm:k-free-ideals-class-and-narrow-class-groups}. \qed

\subsubsection{Proof of Corollary \ref{cor:gegenbauer-number-field}}

Take $\mathfrak{m}=(1)$, so that
$\mathcal{I}_K^{\mathfrak{m}}=\mathcal{I}_K$,
$\Cl_K^{\mathfrak{m}}=\Cl_K$ and the Euler product over
$\mathfrak{p}\mid\mathfrak{m}_0$ in
\eqref{eqn:k-free-ideals-ray-class-asymptotic} is empty. Every nonzero integral
ideal lies in exactly one ideal class, so
\[
Q_K^k(X)=\sum_{\mathfrak{C}\in\Cl_K}Q_K^k(X;\mathfrak{C}).
\]
Summing the first asymptotic formula of
Theorem~\ref{thm:k-free-ideals-class-and-narrow-class-groups} over the $h_K$
ideal classes and using $\rho_K^{(1)}=\rho_K/h_K$ gives the main term
$h_K\cdot\rho_K X/(h_K\zeta_K(k))=\rho_K X/\zeta_K(k)$; the error terms are
unchanged, since summing $h_K$ estimates only multiplies the implied constants by $h_K$. \qed

\section*{Data availability}

No external datasets were used in this paper. The computational data appearing in the tables of this paper can be reproduced using the accompanying Jupyter notebook, which is publicly available at \cite{k-free-ideals-code}.

\bibliography{ideals-in-ray-classes}

@book {landau1909handbuch,
    AUTHOR = {Landau, Edmund},
     TITLE = {Handbuch der {L}ehre von der {V}erteilung der {P}rimzahlen. 2
              {B}\"ande},
      NOTE = {2d ed,
              With an appendix by Paul T. Bateman},
 PUBLISHER = {Chelsea Publishing Co., New York},
      YEAR = {1953},
     PAGES = {xviii+pp. 1--564; ix+pp. 565--1001},
   MRCLASS = {10.0X},
  MRNUMBER = {68565},
MRREVIEWER = {L.\ Schoenfeld},
}

@online{k-free-ideals-code,
  author  = {Barquero-Sanchez, Adrian and
             Heimrath, Jack and
             Sing, Bernd and
             Sirolli, Nicolás and
             Spivey, Caylee and
             Wijaya, Michael},
  title   = {Computations for \emph{The Distribution of $k$-Free Ideals
             in Ray Class Groups}},
  year    = {2026},
  note    = {SageMath 10.8 code accompanying the paper},
  url     = {https://github.com/adrianbs11/k-free-ideals-ray-classes},
  urldate = {2026-08-04},
}

@article {huxley1994number,
    AUTHOR = {Huxley, M. N. and Watt, N.},
     TITLE = {The number of ideals in a quadratic field},
      NOTE = {K. G. Ramanathan memorial issue},
   JOURNAL = {Proc. Indian Acad. Sci. Math. Sci.},
  FJOURNAL = {Indian Academy of Sciences. Proceedings. Mathematical
              Sciences},
    VOLUME = {104},
      YEAR = {1994},
    NUMBER = {1},
     PAGES = {157--165},
      ISSN = {0253-4142,0973-7685},
   MRCLASS = {11R47 (11L07 11P21 11R42)},
  MRNUMBER = {1280063},
MRREVIEWER = {John\ B.\ Friedlander},
       DOI = {10.1007/BF02830879},
       URL = {https://doi.org/10.1007/BF02830879},
}

@book {iwaniec-kowalski2004,
    AUTHOR = {Iwaniec, Henryk and Kowalski, Emmanuel},
     TITLE = {Analytic number theory},
    SERIES = {American Mathematical Society Colloquium Publications},
    VOLUME = {53},
 PUBLISHER = {American Mathematical Society, Providence, RI},
      YEAR = {2004},
     PAGES = {xii+615},
      ISBN = {0-8218-3633-1},
   MRCLASS = {11-02 (11Fxx 11Lxx 11Mxx 11Nxx)},
  MRNUMBER = {2061214},
MRREVIEWER = {K.\ Soundararajan},
       DOI = {10.1090/coll/053},
       URL = {https://doi.org/10.1090/coll/053},
}

@book{montgomery-vaughan2007,
  author    = {Montgomery, Hugh L. and Vaughan, Robert C.},
  title     = {Multiplicative Number Theory I: Classical Theory},
  series    = {Cambridge Studies in Advanced Mathematics},
  volume    = {97},
  publisher = {Cambridge University Press},
  address   = {Cambridge},
  year      = {2007}
}

@book {narkiewicz1974,
    AUTHOR = {Narkiewicz, W{\l}adys{\l}aw},
     TITLE = {Elementary and analytic theory of algebraic numbers},
    SERIES = {Monografie Matematyczne},
    VOLUME = {Tom 57},
 PUBLISHER = {PWN---Polish Scientific Publishers, Warsaw},
      YEAR = {1974},
     PAGES = {630 pp. (errata insert)},
   MRCLASS = {12-01},
  MRNUMBER = {347767},
MRREVIEWER = {P.\ J.\ Weinberger},
}

@book {landau1949,
    AUTHOR = {Landau, Edmund},
     TITLE = {Einf\"uhrung in die elementare und analytische {T}heorie der
              algebraischen {Z}ahlen und der {I}deale},
 PUBLISHER = {Chelsea Publishing Co., New York},
      YEAR = {1949},
     PAGES = {vii+147},
   MRCLASS = {10.0X},
  MRNUMBER = {31002},
}

@article {nymann1992distribution,
    AUTHOR = {Nymann, J. E.},
     TITLE = {The distribution of relatively {$r$}-prime integers in residue
              classes},
   JOURNAL = {Rocky Mountain J. Math.},
  FJOURNAL = {The Rocky Mountain Journal of Mathematics},
    VOLUME = {22},
      YEAR = {1992},
    NUMBER = {4},
     PAGES = {1473--1482},
      ISSN = {0035-7596,1945-3795},
   MRCLASS = {11N69 (11N25)},
  MRNUMBER = {1201105},
MRREVIEWER = {G\'{e}rald\ Tenenbaum},
       DOI = {10.1216/rmjm/1181072668},
       URL = {https://doi.org/10.1216/rmjm/1181072668},
}

@article {lehmer1900asymptotic,
    AUTHOR = {Lehmer, Derrick Norman},
     TITLE = {Asymptotic {E}valuation of {C}ertain {T}otient {S}ums},
   JOURNAL = {Amer. J. Math.},
  FJOURNAL = {American Journal of Mathematics},
    VOLUME = {22},
      YEAR = {1900},
    NUMBER = {4},
     PAGES = {293--335},
      ISSN = {0002-9327,1080-6377},
   MRCLASS = {DML},
  MRNUMBER = {1505840},
       DOI = {10.2307/2369728},
       URL = {https://doi.org/10.2307/2369728},
}

@article {cohen1963distribution,
    AUTHOR = {Cohen, E. and Robinson, Richard L.},
     TITLE = {On the distribution of the {$k$}-free integers in residue
              classes},
   JOURNAL = {Acta Arith.},
  FJOURNAL = {Polska Akademia Nauk. Instytut Matematyczny. Acta Arithmetica},
    VOLUME = {8},
      YEAR = {1962/63},
     PAGES = {283--293},
      ISSN = {0065-1036},
   MRCLASS = {10.42},
  MRNUMBER = {157960},
MRREVIEWER = {L.\ Mirsky},
       DOI = {10.4064/aa-8-3-283-293},
       URL = {https://doi.org/10.4064/aa-8-3-283-293},
}

@article{benkoski1976probability,
  title={The probability that $k$ positive integers are relatively $r$-prime},
  author={Benkoski, Stanley J.},
  journal={Journal of Number Theory},
  volume={8},
  number={2},
  pages={218--223},
  year={1976},
  publisher={Elsevier}
}

@article {BP10,
    AUTHOR = {Baker, R. C. and Powell, K.},
     TITLE = {The distribution of {$k$}-free numbers},
   JOURNAL = {Acta Math. Hungar.},
  FJOURNAL = {Acta Mathematica Hungarica},
    VOLUME = {126},
      YEAR = {2010},
    NUMBER = {1-2},
     PAGES = {181--197},
      ISSN = {0236-5294},
   MRCLASS = {11N25},
  MRNUMBER = {2593323},
MRREVIEWER = {S. W. Graham},
       DOI = {10.1007/s10474-009-9042-9},
       URL = {https://doi.org/10.1007/s10474-009-9042-9},
}

@article {Erd60,
    AUTHOR = {Erd\H{o}s, P.},
     TITLE = {\"{U}ber die kleinste quadratfreie {Z}ahl einer arithmetischen
              {R}eihe},
   JOURNAL = {Monatsh. Math.},
  FJOURNAL = {Monatshefte f\"{u}r Mathematik},
    VOLUME = {64},
      YEAR = {1960},
     PAGES = {314--316},
      ISSN = {0026-9255},
   MRCLASS = {10.00},
  MRNUMBER = {118705},
MRREVIEWER = {L. Mirsky},
       DOI = {10.1007/BF01498607},
       URL = {https://doi.org/10.1007/BF01498607},
}

@article {GMRR21,
    AUTHOR = {Gorodetsky, Ofir and Matom\"{a}ki, Kaisa and Radziwi\l\l , Maksym and
              Rodgers, Brad},
     TITLE = {On the variance of squarefree integers in short intervals and
              arithmetic progressions},
   JOURNAL = {Geom. Funct. Anal.},
  FJOURNAL = {Geometric and Functional Analysis},
    VOLUME = {31},
      YEAR = {2021},
    NUMBER = {1},
     PAGES = {111--149},
      ISSN = {1016-443X},
   MRCLASS = {11N25},
  MRNUMBER = {4244849},
MRREVIEWER = {Tsz Ho Chan},
       DOI = {10.1007/s00039-021-00557-5},
       URL = {https://doi.org/10.1007/s00039-021-00557-5},
}

@article {Gege85,
    AUTHOR = {Gegenbauer, Leopold},
     TITLE = {{Asymptotische Gesetze der Zahlentheorie}},
   JOURNAL = {Denkschriften
Akad. Wiss. Wien},
    VOLUME = {49},
      YEAR = {1885},
     PAGES = {37--80},
}

@article {Hoo75,
    AUTHOR = {Hooley, C.},
     TITLE = {A note on square-free numbers in arithmetic progressions},
   JOURNAL = {Bull. London Math. Soc.},
  FJOURNAL = {The Bulletin of the London Mathematical Society},
    VOLUME = {7},
      YEAR = {1975},
     PAGES = {133--138},
      ISSN = {0024-6093},
   MRCLASS = {10A25},
  MRNUMBER = {371799},
MRREVIEWER = {Donald L. Goldsmith},
       DOI = {10.1112/blms/7.2.133},
       URL = {https://doi.org/10.1112/blms/7.2.133},
}

@book {Lan94,
    AUTHOR = {Lang, Serge},
     TITLE = {Algebraic number theory},
    SERIES = {Graduate Texts in Mathematics},
    VOLUME = {110},
   EDITION = {Second},
 PUBLISHER = {Springer-Verlag, New York},
      YEAR = {1994},
     PAGES = {xiv+357},
      ISBN = {0-387-94225-4},
   MRCLASS = {11Rxx (11-01 11-02)},
  MRNUMBER = {1282723},
MRREVIEWER = {M. Ram Murty},
       DOI = {10.1007/978-1-4612-0853-2},
       URL = {https://doi.org/10.1007/978-1-4612-0853-2},
}

@article {Liu14,
    AUTHOR = {Liu, H.-Q.},
     TITLE = {On the distribution of {$k$}-free integers},
   JOURNAL = {Acta Math. Hungar.},
  FJOURNAL = {Acta Mathematica Hungarica},
    VOLUME = {144},
      YEAR = {2014},
    NUMBER = {2},
     PAGES = {269--284},
      ISSN = {0236-5294},
   MRCLASS = {11L07 (11B83)},
  MRNUMBER = {3274401},
MRREVIEWER = {Ping Ding},
       DOI = {10.1007/s10474-014-0454-9},
       URL = {https://doi.org/10.1007/s10474-014-0454-9},
}

@article {Liu16,
    AUTHOR = {Liu, H.-Q.},
     TITLE = {On the distribution of squarefree numbers},
   JOURNAL = {J. Number Theory},
  FJOURNAL = {Journal of Number Theory},
    VOLUME = {159},
      YEAR = {2016},
     PAGES = {202--222},
      ISSN = {0022-314X},
   MRCLASS = {11L07 (11B83)},
  MRNUMBER = {3412720},
MRREVIEWER = {Victor Zhenyu Guo},
       DOI = {10.1016/j.jnt.2015.07.013},
       URL = {https://doi.org/10.1016/j.jnt.2015.07.013},
}

@article {MOT21,
    AUTHOR = {Mossinghoff, Michael J. and Oliveira e Silva, Tom\'{a}s and
              Trudgian, Timothy S.},
     TITLE = {The distribution of {$k$}-free numbers},
   JOURNAL = {Math. Comp.},
  FJOURNAL = {Mathematics of Computation},
    VOLUME = {90},
      YEAR = {2021},
    NUMBER = {328},
     PAGES = {907--929},
      ISSN = {0025-5718},
   MRCLASS = {11N25 (11M26 11N60 11Y35)},
  MRNUMBER = {4194167},
MRREVIEWER = {D. R. Heath-Brown},
       DOI = {10.1090/mcom/3581},
       URL = {https://doi.org/10.1090/mcom/3581},
}

@article{Man21,
    AUTHOR = {Mangerel, Alexander P.},
     TITLE = {Squarefree integers in arithmetic progressions to smooth
              moduli},
   JOURNAL = {Forum Math. Sigma},
  FJOURNAL = {Forum of Mathematics. Sigma},
    VOLUME = {9},
      YEAR = {2021},
     PAGES = {Paper No. e72, 47},
   MRCLASS = {11N25 (11N37 11N69)},
  MRNUMBER = {4332498},
MRREVIEWER = {Donald Jason Gibson},
       DOI = {10.1017/fms.2021.67},
       URL = {https://doi.org/10.1017/fms.2021.67},
}

@article {Nun15,
    AUTHOR = {Nunes, Ramon M.},
     TITLE = {Squarefree numbers in arithmetic progressions},
   JOURNAL = {J. Number Theory},
  FJOURNAL = {Journal of Number Theory},
    VOLUME = {153},
      YEAR = {2015},
     PAGES = {1--36},
      ISSN = {0022-314X},
   MRCLASS = {11A05 (11B25 11L03)},
  MRNUMBER = {3327562},
MRREVIEWER = {S\'{a}ndor Z. Kiss},
       DOI = {10.1016/j.jnt.2014.12.025},
       URL = {https://doi.org/10.1016/j.jnt.2014.12.025},
}

@article {Nun22,
    AUTHOR = {Nunes, Ramon M.},
     TITLE = {Moments of the distribution of {$k$}-free numbers in short
              intervals and arithmetic progressions},
   JOURNAL = {Bull. Lond. Math. Soc.},
  FJOURNAL = {Bulletin of the London Mathematical Society},
    VOLUME = {54},
      YEAR = {2022},
    NUMBER = {4},
     PAGES = {1282--1298},
      ISSN = {0024-6093},
   MRCLASS = {11N37 (11N69)},
  MRNUMBER = {4474124},
MRREVIEWER = {Ofir Gorodetsky},
       DOI = {10.1112/blms.12628},
       URL = {https://doi.org/10.1112/blms.12628},
}

@article {Orr71,
    AUTHOR = {Orr, Richard C.},
     TITLE = {Remainder estimates for squarefree integers in arithmetic
              progression},
   JOURNAL = {J. Number Theory},
  FJOURNAL = {Journal of Number Theory},
    VOLUME = {3},
      YEAR = {1971},
     PAGES = {474--497},
      ISSN = {0022-314X},
   MRCLASS = {10.42},
  MRNUMBER = {286764},
MRREVIEWER = {G. Greaves},
       DOI = {10.1016/0022-314X(71)90015-1},
       URL = {https://doi.org/10.1016/0022-314X(71)90015-1},
}

@article {Pra58,
    AUTHOR = {Prachar, Karl},
     TITLE = {\"{U}ber die kleinste quadratfreie {Z}ahl einer arithmetischen
              {R}eihe},
   JOURNAL = {Monatsh. Math.},
  FJOURNAL = {Monatshefte f\"{u}r Mathematik},
    VOLUME = {62},
      YEAR = {1958},
     PAGES = {173--176},
      ISSN = {0026-9255},
   MRCLASS = {10.0X},
  MRNUMBER = {92806},
MRREVIEWER = {H. N. Shapiro},
       DOI = {10.1007/BF01301288},
       URL = {https://doi.org/10.1007/BF01301288},
}

@misc{SageMathCoCalc,
  author       = {{The Sage Developers}},
  title        = {{SageMath, the Sage Mathematics Software System (Version 10.8)}},
  year         = {2026},
  url          = {https://www.sagemath.org},
  note         = {Computations carried out using SageMath via the CoCalc platform: \url{https://cocalc.com}},
}

@article{Sit10,
    AUTHOR = {Sittinger, Brian D.},
     TITLE = {The probability that random algebraic integers are relatively
              {$r$}-prime},
   JOURNAL = {J. Number Theory},
  FJOURNAL = {Journal of Number Theory},
    VOLUME = {130},
      YEAR = {2010},
    NUMBER = {1},
     PAGES = {164--171},
      ISSN = {0022-314X},
   MRCLASS = {11R04 (11K99 11R42)},
  MRNUMBER = {2569847},
MRREVIEWER = {Zhi-Wei Sun},
       DOI = {10.1016/j.jnt.2009.06.008},
       URL = {https://doi.org/10.1016/j.jnt.2009.06.008},
}

@article{Sit20,
    AUTHOR = {Sittinger, Brian D. and DeMoss, Ryan D.},
     TITLE = {The probability that ideals in a number ring are {$k$}-wise
              relatively {$r$}-prime},
   JOURNAL = {Int. J. Number Theory},
  FJOURNAL = {International Journal of Number Theory},
    VOLUME = {16},
      YEAR = {2020},
    NUMBER = {8},
     PAGES = {1753--1765},
      ISSN = {1793-0421},
   MRCLASS = {11N25 (11R04 11R42 11R45)},
  MRNUMBER = {4143683},
MRREVIEWER = {W. Narkiewicz},
       DOI = {10.1142/S1793042120500918},
       URL = {https://doi.org/10.1142/S1793042120500918},
}

@online{Sut21,
  author  = {Sutherland, Andrew V.},
  title   = {18.785 {N}umber {T}heory {I}},
  year    = {2021},
  url     = {https://math.mit.edu/classes/18.785/2021fa/index.html},
  urldate = {2026-07-09},
  note    = {Lecture notes, Fall 2021, Massachusetts Institute of Technology},
}
\bibliographystyle{alphaurl}

\end{document}